\documentclass{article}
\usepackage[greek, english]{babel}
\usepackage{amsmath}
\usepackage{amssymb}
\usepackage{eufrak}
\usepackage{eurosym}
\usepackage{epsf}
\usepackage{epsfig}
\usepackage{color}

\newfont{\gothic}{eufm10}

\def\Z{{\mathbb{Z}}}                   \def\R{{\RR}} \def\Q{{\mathbb{Q}}}
\def\RR{{\mathbb{R}}}        \def\N{{\mathbb{N}}}

\newtheorem{theorem}{Theorem}[section]
\newtheorem{lemma}[theorem]{Lemma}
\newtheorem{proposition}[theorem]{Proposition}

\newtheorem{definition1}[theorem]{Definition}
\newenvironment{definition}{\begin{definition1}\rm}{\hfill $\triangle$\end{definition1}}
\newenvironment{proof}{\addvspace\baselineskip\noindent{\it
Proof:}}{\hspace*{\fill}         $\Box$\par\addvspace\baselineskip}
 
\newenvironment{proofof}[1]{\addvspace\baselineskip\noindent{\it Proof
{\bf (of #1)}:}}{\hspace*{\fill} $\Box$\par\addvspace\baselineskip}

\newtheorem{remark1}[theorem]{Remark}

\newtheorem{example1}[theorem]{Example}
\newenvironment{example}{\begin{example1}\rm }{\hfill $\triangle$\end{example1}}

\def\barray{\begin{eqnarray*}}             \def\earray{\end{eqnarray*}}
\def\beq{\begin{equation}} \def\eeq{\end{equation}}

\makeatletter
\title{The laminations of a crystal near an anti-continuum limit}
\author{Vincent Knibbeler\thanks{Mathematical Modelling Lab, Northumbria University, Newcastle, UK, $\hfill$
{\tt vincent.knibbeler@northumbria.ac.uk}.}, Blaz Mramor\thanks{Institute of Mathematics, Albert-Ludwigs-Universit\"at, Freiburg, Germany, {\tt blazmramor@hotmail.com}.} and Bob Rink\thanks{Department of Mathematics, VU University, Amsterdam, The Netherlands, {\tt b.w.rink@vu.nl}.}} 
\date{\today}

\begin{document}\hyphenation{boun-da-ry mo-no-dro-my sin-gu-la-ri-ty ma-ni-fold ma-ni-folds re-fe-rence se-cond se-ve-ral dia-go-na-lised con-ti-nuous thres-hold re-sul-ting fi-nite-di-men-sio-nal ap-proxi-ma-tion pro-per-ties ri-go-rous mo-dels mo-no-to-ni-ci-ty pe-ri-o-di-ci-ties mi-ni-mi-zer mi-ni-mi-zers know-ledge ap-proxi-mate pro-per-ty con-fi-gu-ra-tions dif-fe-ren-tiable Non-li-neari-ty diffe-ren-tial con-fi-gu-ra-tion non-de-ge-ne-rate pro-blem follo-wing la-mi-na-tions}
\newcommand{\X}{\mathbb{X}}
\newcommand{\p}{\partial}
\maketitle
\begin{abstract}\noindent 
The anti-continuum limit of a monotone variational recurrence relation consists of a lattice of uncoupled particles in a periodic background. This limit supports many trivial equilibrium states that persist as solutions of the model with small coupling. We investigate when a persisting solution generates a so-called lamination and prove that near the anti-continuum limit the collection of laminations of solutions is homeomorphic to the $(N-1)$-dimensional simplex, with $N$ the number of distinct local minima of the background potential. This generalizes a result by Baesens and MacKay on twist maps near an anti-integrable limit.
\end{abstract}

\section{Introduction}
In this note we shall prove that certain models from statistical mechanics possess a continuous family of laminations of solutions when they are close to a nondegenerate anti-continuum limit. The models for which this is true are ``monotone variational recurrence relations'' of the form
\begin{align}\label{FKepsilon} V' (x_i) + \varepsilon \sum_{j\in \Z^d} \p_i S_j(x) = 0\ \mbox{for} \ i\in \Z^d \ \mbox{and} \ x_i\in\R\, .\end{align} 
A solution to such a model is a map $x:\Z^d\to \R$ and shall be called a {\it stationary configuration}. The space of all configurations will be denoted $\R^{\Z^d}$.

Equations (\ref{FKepsilon}) describe a lattice of particles in a periodic background that experience ferromagnetic attraction. Thus, the background potential $V:\R\to \R$ is assumed one-periodic. The properties required of the local interaction potentials $S_j(x)$ are the usual ones of Aubry-Mather theory. They will be made precise in Section \ref{setting} and guarantee that (\ref{FKepsilon}) is a well-defined monotone recurrence relation. 

Equations of the type (\ref{FKepsilon}) have obtained quite some attention in recent years, most notably in the work of De la Llave et al. \cite{candel-llave}, \cite{llave-valdinoci07}, \cite{llave-valdinoci}, \cite{llave-lattices}. In contrast with these studies, we will assume in this paper that the parameter $0 \leq \varepsilon \ll 1$ is small, so that (\ref{FKepsilon}) describes a weak interaction between the particles on the lattice.  %A solution to equations (\ref{FKepsilon}) is called a and we shall denote the space of configurations  by $\R^{\Z^d}$.  %In this context, the more dimensional real-valued sequences $x\in \R^{\Z^d}$ are usually called ``configurations''. 

%A prototypical example of such a recurrence relation is the famous Frenkel-Kontorova problem: 
%\begin{align}\label{FKepsilon1}  V' (x_i) = \varepsilon (\Delta x)_i\ \mbox{for} \ i\in \Z^d \ \mbox{and} \ x_i\in\R\, .\end{align} 
%The operator $\Delta$ denotes the discrete Laplacian defined by $(\Delta x)_i:= \sum_{||j-i||=1} (x_j-x_i)$ with $||j-i||:=\sum_{k=1}^d |j_k-i_k|$. 

In the special case that $\varepsilon=0$ equation (\ref{FKepsilon}) reduces to 
 \begin{align}\label{antilimit}V'(x_i)=0 \ \mbox{for all}\ i\in\Z^d\, . \end{align} 
We shall refer to (\ref{antilimit}) as the {\it anti-continuum limit} of (\ref{FKepsilon}) because it models a medium in which the interaction between particles is absent. The solutions of the anti-continuum limit are simply those configurations $x$ for which every $x_i$ is a critical point of $V$. If one assumes that $V$ is a Morse function and that $x$ is a solution to (\ref{antilimit}) of bounded oscillation, then one can prove \cite{abramovici}, \cite{baesens}, \cite{MackayMeiss} that $x$ can be continued to a solution $x(\varepsilon)$ of (\ref{FKepsilon}) for $0< \varepsilon\ll 1$. This is of great help in understanding many of the solutions to (\ref{FKepsilon}) that exist close to the anti-continuum limit.

This paper is concerned with classifying the laminations of solutions of (\ref{FKepsilon}) close to the anti-continuum limit. Laminations are defined as follows:

\begin{definition}
We say that a collection $\Gamma\subset \R^{\Z^d}$ is a {\it lamination} if it is nonempty, ordered, translation-invariant and if in the topology of pointwise convergence it is closed, minimal and homeomorphic to a Cantor set.
\end{definition}
The precise meaning of the words ``ordered'', ``translation'' and ``minimal'' will be explained later. It is well-known \cite{candel-llave}, \cite{MramorRink1} that any lamination $\Gamma$ has an irrational {\it rotation vector}, i.e. there exists a unique $\omega\in \R^d\backslash \Q^d$ such that for any $x\in \Gamma$ and any $i\in \Z^d$ the limit
$$\lim_{n\to\infty} \frac{x_{ni}}{n} \ \mbox{exists and is equal to}\ \omega \cdot i\, .$$
Examples of laminations of solutions to (\ref{FKepsilon}) are the well-known Aubry-Mather sets of irrational rotation vectors \cite{AubryDaeron}, \cite{candel-llave},  \cite{MatherTopology}, \cite{MramorRink1}. It was proved by Bangert \cite{bangert87} that the Aubry-Mather set of rotation vector $\omega\in \R^d\backslash \Q^d$ is the unique lamination that has this rotation vector and that consists entirely of global minimizers. Nevertheless, the Aubry-Mather set may not be the only lamination of solutions of rotation vector $\omega$ though. This fact is illustrated by Theorem \ref{mainthm} below.  %although only one of them can consist of action-minimizers, it is in principle possible for equation (\ref{FKepsilon}) to have several laminations of solutions of rotation vector $\omega$. %
It is the main result of this paper and it describes the set of all laminations of solutions to (\ref{FKepsilon}) of a fixed irrational rotation vector close to the anti-continuum limit. To formulate it, let us say that two points $\sigma_1, \sigma_2 \in \R$ are geometrically distinct if $\sigma_1-\sigma_2\notin \Z$. Clearly, if $V$ is a one-periodic Morse function, then it has finitely many geometrically distinct critical points.  %Our result is the following:

\begin{theorem}\label{mainthm}
Let $V$ be a one-periodic Morse function with $N$ geometrically distinct local minima and let $\omega\in \R^d\backslash \Q^d$. Then there is an $\varepsilon_1>0$ such that for all $0<\varepsilon\leq \varepsilon_1$ the collection of laminations for (\ref{FKepsilon}) of rotation vector $\omega$ is homeomorphic  to the $(N-1)$-dimensional simplex $$\Delta_{N-1} :=\left\{ p\in\R^N\ | \ p_1+ \ldots +p_N=1\, \mbox{and} \ p_j\geq 0 \ \mbox{for all} \ 1\leq j \leq N \right\}\, .$$
\end{theorem}
An element $p\in\Delta_{N-1}$ has the interpretation of a probability distribution or density of states. In fact, we will show that if $0\leq \sigma_1<\ldots< \sigma_N<1$ are the geometrically distinct local minima of $V$, then for every $p\in \Delta_{N-1}$ there is exactly one lamination of solutions to (\ref{FKepsilon}) of which the elements have probability $p_j$ of lying near one of the local minima in $\sigma_j+\Z$. The set of laminations is given a topology that captures precisely this statistical information. This so-called {\it vague topology} is quite well-known in Aubry-Mather theory \cite{baesens}, \cite{moredenjoy} and will be defined in Section \ref{measuresection}. %The theorem means that the collection of laminations of solutions to (\ref{FKepsilon}) of an irrational rotation vector $\omega$ is an $(N-1)$-parameter family of which the Aubry-Mather set of rotation vector $\omega$ is one particular element.

Theorem \ref{mainthm} is a generalization of a result of Baesens and MacKay \cite{baesens}, who proved that for every $\omega\in \R\backslash \Q$ a Hamiltonian twist map of $\R\times\R/\Z$ near a nondegenerate anti-integrable limit possesses an $(N-1)$-dimensional family of remnant circles of rotation number $\omega$. We shall describe this result, as well as a few other facts from the theory of twist maps, in Appendix \ref{appendix}. The result in \cite{baesens} arises as a special case of Theorem \ref{mainthm}, because the orbits of a Hamiltonian twist map are in one-to-one correspondence with the solutions of a second order recurrence relation of the form
\begin{align}\label{FKepsilonmap}
V'(x_i) + \varepsilon \left( \p S(x_{i-1},x_i)/\p x_i+\p S(x_i, x_{i+1})/\p x_i \right) = 0\ \mbox{for} \ i \in \Z \ \mbox{and}\ x_i\in \R \, .
\end{align}
Indeed, equation (\ref{FKepsilon}) reduces to (\ref{FKepsilonmap}) in case that $d=1$ and $S_j(x) = \frac{1}{2}S(x_{j-1}, x_j)+\frac{1}{2}S(x_{j}, x_{j+1})$. The function $V(x_j)+\varepsilon S(x_{j-1}, x_j):\R^2\to\R$ is the so-called generating function of the twist map. 

%An example of such a twist map is the well-known Chirikov standard map, which is the quotient to $\R\times\R/\Z$ of the map
%\begin{align}\label{smap1}
%S_{\varepsilon}: (x,y)\mapsto (x+y+\varepsilon^{-1} V'(x), y+\varepsilon^{-1} V'(x)) \ \mbox{on} \ \R\times\R.
%\end{align}
%A sequence $i\mapsto (x_i, y_i)$ in $\R\times \R$ is an orbit of (\ref{smap1}) if and only if $y_i=x_i-x_{i-1}$ and 
%\begin{align}\label{FKmap}
%\varepsilon(x_{i+1}-x_i+x_{i-1}) + V'(x_i)=0\ \mbox{for all} \ i\in \Z\, .
%\end{align}
%Thus, the map $S_{\varepsilon}$ is equivalent to the Frenkel-Kontorova problem in dimension $d=1$ and the result in \cite{baesens} is the one-dimensional equivalent of Theorem \ref{mainthm}. 

Although it may seem that the generalization of the result in \cite{baesens} to equations of the form (\ref{FKepsilon}) is straightforward, this turns out not to be the case. To understand this, one should realize that the laminations of solutions to (\ref{FKepsilon}) consist of local action-minimizers. A variant of Aubry's crossing lemma says that the local minimizers of a one-dimensional second order recurrence relation such as (\ref{FKepsilonmap}) can cross at most once. This fact is at the core of the proof given in \cite{baesens}. Nevertheless, Aubry's crossing lemma does not hold for recurrence relations of higher order or in higher dimensions and this is what makes the proof of Theorem \ref{mainthm} more delicate than the proof given in \cite{baesens}. In fact, our proof of Theorem \ref{mainthm} is based on ideas from Bangert \cite{bangert87}. We moreover remark that the proofs in \cite{baesens}, but also those in \cite{MackayMeiss}, seem to contain a few imperfections that we fix in this paper.

% In fact, the proof in \cite{baesens} is facilitated by the fact that equation (\ref{FKepsilonmap}) is a one-dimensional recurrence relation with nearest neighbor interaction. Together, these two specific properties imply that the local minimizers of (\ref{FKepsilonmap}) can cross only once and as a consequence it follows directly that the collection of recurrent local minimizers of (\ref{FKepsilonmap}) of rotation number $\omega\in \R\backslash \Q$ is ordered.  
 
 %As a consequence, it is actually valid for quite a general class of mathematical problems with a monotone variational structure. This class includes lattice recurrence relations near a nondegenerate anti-continuum limit of the form 
%$$ \varepsilon \!\!\! \sum_{||j-i||\leq r}\!\!\! \p_iS_j(x) = V'(x_i) \ \mbox{for}\ i\in \Z^d\, .$$ 
%Here, the functions $S_j=S_j(x)$ are so-called twist potentials, i.e. they should satisfy assumptions of the kind {\bf A} - {\bf E} in \cite{MramorRink1}. 

%Theorem \ref{mainthm} may remain valid for monotone variational recurrence relations on more general lattices than $\Z^d$, including Cayley graphs and Bethe lattices \cite{candel-llave}, \cite{llave-valdinoci}, but we did not try to formulate or prove any of the corresponding statements here.

This paper is organized as follows. In Section \ref{setting} we specify our requirements on the interaction potentials $S_j(x)$. In Section \ref{continuationsection} we show that   solutions of the anti-continuum limit (\ref{antilimit}) of bounded oscillation persist to solutions of (\ref{FKepsilon}) for $0<\varepsilon \ll 1$. In Section \ref{birkhoffsection} we review some well-known facts about ordered configurations and in particular, we define and study laminations. In Sections \ref{propertiessection} and \ref{minimizersection} we then prove two important crossing properties of recurrent Birkhoff configurations near the anti-continuum limit. %These properties precisely determine the recurrent Birkhoff configurations near the anti-integrable limit. 
These properties form the main ingredients of the proof of Theorem \ref{mainthm} that we give in Section \ref{proofsection}. Finally, in Appendix \ref{proofappendix} we prove some rather well-known properties of Birkhoff configurations, while in Appendix \ref{appendix} we recall some facts from the theory of Hamiltonian twist maps and relate them to the results in this paper.

\section{Monotone variational recurrence relations}\label{setting}

Recall that we are interested in recurrence relations of the form (\ref{FKepsilon}), i.e. 
\begin{align} \label{FKepsilon2} V' (x_i) + \varepsilon R_i(x) = 0 \ \mbox{for} \ i\in \Z^d,\, x_i\in\R \ \mbox{and}\ R_i(x):=\sum_{j\in \Z^d} \p_i S_j(x)  \, .\end{align} 
Throughout this text we shall assume that the background potential $V:\R\to\R$ is a twice continuously differentiable and one-periodic Morse function. The functions $R_i:\R^{\Z^d}\to \R$ are meant to describe the ferromagnetic interactions between the particles in the lattice. It is clear that equation (\ref{FKepsilon2}) describes the stationary points of an ``action functional'' $W^{\varepsilon}: \R^{\Z^d}\to\R$ defined by
$$W^{\varepsilon}(x):=  \sum_{j\in \Z^d} S_j^{\varepsilon}(x)\ \mbox{with}\ S_j^{\varepsilon}(x):=V(x_j) + \varepsilon S_j(x)\, .$$
Indeed, the derivative of $W^{\varepsilon}(x)$ with respect to $x_i$ is precisely the left hand side of (\ref{FKepsilon2}). In order to guarantee that (\ref{FKepsilon2}) determines a well-defined monotone recurrence relation, we shall impose some conditions on the local interaction potentials $S_j(x)$ below.

To formulate the first condition, let $r>0$ be an integer, define 
$$||i||:=\sum_{k=1}^d |i_k| \ \mbox{for}\ i\in \Z^d \ \mbox{and let}\ B_r(j):=\{k\in \Z^d\ | \ ||k-j||\leq r\}$$ 
be the ball of radius $r$ centered at $j\in \Z^d$. Assume now that a smooth function $s_j:\R^{B_r(j)}\to\R$ is given. Then we can define a function $S_j:\R^{\Z^d}\to\R$ by setting $S_j(x):=s_j(x|_{B_r(j)})$. This just means that $S_j$ depends only on the finitely many variables $x_k$ with $||k-j||\leq r$. 

By construction, such an $S_j$ is continuous with respect to the topology of pointwise convergence:  if $x^n, x\in\R^{\Z^d}$ and $\lim_{n\to\infty} x^n=x$ pointwise, then obviously also $\lim_{n\to\infty} S_j(x^n)=S_j(x)$. Moreover, if the $s_j$ are continuously differentiable, then so are the $S_j$. Now we can formulate our first condition. 
\begin{itemize}
\item[{\bf A.}] The functions $S_j$ are of finite range and twice continuously differentiable. That is, there is an $0<r<\infty$ and for every $j\in\Z^d$ there is a twice continuously differentiable function $s_j: \R^{B_r(j)}\to \R$ such that $S_j(x)=s_j(x|_{B_r(j)})$. 
\end{itemize}
%In other words, the function $S_j$ depends only on the finitely many variables $x_k$ with $||k-j||\leq r$. 
We think of $r$ as the finite range of interaction of our lattice. Most importantly, 
even though the formal sum $W^{\varepsilon}(x)=\sum_{j\in\Z^d} S^{\varepsilon}_j(x)$ will generally be divergent, condition {\bf A} ensures that  
$$ R_i(x) =   \sum_{j\in \Z^d} \p_i S_j(x) = \sum_{||j-i||\leq r} \p_i S_j(x) $$
is a finite sum. Thus, condition {\bf A} guarantees that (\ref{FKepsilon2}) is a well-defined recurrence relation. 

Another noteworthy consequence of condition {\bf A} is that the set of solutions to (\ref{FKepsilon2}) is closed under pointwise convergence: if $x^1, x^2, \ldots$ are solutions to (\ref{FKepsilon2}) and $\lim_{n\to\infty} x^{n}=x$ pointwise, then also $x$ solves (\ref{FKepsilon2})

To formulate condition {\bf B}, we introduce an action of $\mathbb{Z}^d\times\mathbb{Z}$ on $\mathbb{R}^{\mathbb{Z}^d}$ by ``translations'':
\begin{definition}\label{taudef}
Let $k\in \mathbb{Z}^d$ and $l\in \mathbb{Z}$. The {\it translation} operator $\tau_{k,l}:\mathbb{R}^{\mathbb{Z}^d}\to \mathbb{R}^{\mathbb{Z}^d}$ is defined by
$$(\tau_{k,l} x)_i := x_{i+k}+l\, . $$
\end{definition}
\noindent The graph of $\tau_{k,l}x$, viewed as a subset of $\mathbb{Z}^d\times \mathbb{R}$, is obtained by translating the graph of $x$ over the integer vector $(-k,l)$. This explains our terminology.
\begin{itemize}
\item[{\bf B.}] The functions $S_j$ are translation-invariant: $S_j(\tau_{k,l} x) = S_{j+k}(x)$ for all $j$,  $k$ and $l$.
\end{itemize}
Invariance of  $S_j$ under $\tau_{0,1}$ just means that $S_j(x)=S_j(x+1^{\Z^d})$ for all $j$, that is $S_j$ descends to a function on $\mathbb{R}^{\mathbb{Z}^d}/\mathbb{Z}$. The invariance of the $S_j$ under the shifts $\tau_{k,0}$ expresses the spatial homogeneity of the local interaction potentials. In fact, once one of the $S_j$ is given, for instance $S_0$, then all the others are determined by it. 

Condition {\bf B} guarantees that the set of solutions to (\ref{FKepsilon2}) is translation-invariant: if $x$ solves (\ref{FKepsilon2}), then so does $\tau_{k,l}x$.
%\indent The next condition ensures the growth of the $S_j$ at infinity:
%\begin{itemize}
%\item[{\bf C.}] The functions $S_j$ are bounded from below and coercive: for all $k$ with $||k-j||=1$, $$\lim_{|x_{k}-x_{j}| \to \infty}  S_j(x) = \infty\ .$$
%\end{itemize}
%Condition C says that every function $x\mapsto S_j(x)$ is as coercive as it can possibly be under the restriction that it satisfies the periodicity condition $S_j(\tau_{0,1}x)=S_j(x)$. 

Finally, condition {\bf C} is the most essential one:
\begin{itemize}
\item[{\bf C.}] The functions $S_j$ satisfy a {\it monotonicity condition}:
$$\partial_{i,k}S_j  \leq 0 \ \mbox{for all} \ j \ \mbox{and all} \ i\neq k, \mbox{while} \ \partial_{i,k}S_i  < 0 \ \mbox{for all} \ ||i-k||=1 \, .$$
\end{itemize}
Condition {\bf C} is also called a {\it twist condition} or {\it ferromagnetic condition}. It ensures that equation (\ref{FKepsilon2}) is monotone - in the sense that the partial derivatives of its left hand side with respect to any $x_k$ with $k\neq i$, is nonpositive. 

More importantly, condition {\bf C} makes that equation (\ref{FKepsilon2}) satisfies a comparison principle: when $\varepsilon>0$ and $x\neq y$ are solutions to (\ref{FKepsilon2}) such that $x_i \leq y_i$ for all $i\in \Z^d$, then one can show that actually $x_i<y_i$ for all $i\in \Z^d$. A proof of this fact is given in Lemma \ref{maximumprinciple} in Appendix \ref{proofappendix}. %: if $x$ and $y$ both solve (\ref{FKepsilon2}), then so do the configurations $x\wedge y$ and $x\vee y$ defined by 
%$$(x\wedge y)_i:=\min\{x_i, y_i\}\ \mbox{and} \ (x \vee y)_i:=\max\{x_i, y_i\}\, .$$

%\indent 
%For technical reasons we will also assume:
%\begin{itemize}
%\item[{\bf E.}] The $S_j$ have uniformly bounded second derivatives: there is a constant $C$ such that
%$$|\partial_{i, k}S_j | \leq C\ \mbox{for all} \ i, j, k. $$
%\end{itemize}

 In classical Aubry-Mather theory, one often also imposes that the $S_j$ are bounded from below and grow at infinity. We do not need to require such a coercivity condition in this paper.
 \begin{example}
As an example, let us choose for $S_j(x)$ the harmonic local energy 
$$S_j(x):=\frac{1}{4}\sum_{||k-j||=1}(x_k-x_j)^2\, . $$
It is clear that these local interaction potentials satisfy conditions {\bf A}-{\bf C}. In this case, equation (\ref{FKepsilon2}) reduces to the discrete nonlinear Laplace equation
$$V'(x_i)-\varepsilon (\Delta x)_i =0 \ \mbox{with}\ (\Delta x)_i := \sum_{||j-i||=1}\!\!\!\! (x_j-x_i)\, .$$
%look at the discrete laplacian and observe that $$(\Delta x)_i=\frac{\p}{\p x_i}\left(\frac{1}{8d}\sum_{\|j-i\|=1}(x_i-x_j)^2\right) \ .$$ It is then easy to see that $$(\Delta x)_i=\sum_{j\in \Z^d} \frac{\p}{\p x_i} S_j(x) \ , \ \mbox{where} \  S_j(x):=\frac{1}{8d}\sum_{\|j-i\|=1}(x_j-x_i)^2\ . $$
This equation is also known as the {\it lattice Frenkel-Kontorova model}. %Equation (\ref{FKepsilon2}) can be thought of as a generalization of this discrete nonlinear Laplace equation.
\end{example}

\section{Continuation from the anti-continuum limit}\label{continuationsection}
In this section, we show that certain solutions of the anti-continuum limit (\ref{antilimit}) persist to form solutions of (\ref{FKepsilon2}) for $0<\varepsilon \ll 1$. The first result is Theorem \ref{continuation} below, a variant of which can be found in \cite{MackayMeiss} in the context of Hamiltonian twist maps close to an anti-integrable limit.

Although they are quite elementary, we spell out the details of the proof of Theorem \ref{continuation}, because we will need these later. We denote by 
$$||x||_{\infty}:=\sup_{i\in\Z^d}|x_i| \ \mbox{and} \ B_{\delta}(x):=\{X\in\R^{\Z^d}\ | \ ||X-x||_{\infty}< \delta \}$$
respectively the supremum-norm of a configuration $x$ and the supremum-$\delta$-ball around $x$.

\begin{theorem}[Continuation] \label{continuation} 
Let $x:\Z^d\to\R$ be a solution to (\ref{antilimit}) with the property that $${\rm osc}(x) := \sup_{||j-i||\leq r}\!\! |x_j-x_i| \leq K <\infty\, .$$
Then there exist an $\varepsilon_0>0$ and a $\delta_0>0$, depending only on $K$, such that for every $0\leq \varepsilon\leq \varepsilon_0$ there is a unique $x(\varepsilon)\in B_{\delta_0}(x)$ that satisfies (\ref{FKepsilon2}). It moreover holds that $\lim_{\varepsilon\searrow 0} ||x(\varepsilon)-x||_{\infty}=0$.
\end{theorem}
\begin{proof} A sequence $X\in \R^{\Z^d}$ is a solution to (\ref{FKepsilon2}) if and only if 
\begin{align}\label{defF}
F(X, \varepsilon)=0, \ \mbox{where}\ F:\R^{\Z^d} \times \R \to \R^{\Z^d} \ \mbox{is defined by} \ F(X, \varepsilon)_i:=V'(X_i)+\varepsilon R_i(X) \, .
\end{align}
In particular, $F(x,0)=0$ for any $x$ that satisfies (\ref{antilimit}). One now wants to apply the implicit function theorem to conclude the existence of a family  $x(\varepsilon)$ near $x$ with $F(x(\varepsilon), \varepsilon)=0$. We will explicitly construct and investigate the contraction operator $K_{\varepsilon, x}$ that is used to find $x(\varepsilon)$. 

As a first important remark, let us note that the set $$\{X:\Z^d\to \R\, |\, {\rm osc}(X) \leq K+1 \}/\Z \subset \R^{\Z^d}/ \Z$$ is compact in the topology of pointwise convergence. As a consequence, the continuous functions $\p_{i}S_j$ are bounded on this set. In view of assumption {\bf B} it holds that $\p_{i-k}S_j(\tau_{k,l}x) = \p_{i}S_{j+k}(x)$ and therefore we have that there exists a uniform constant $C_1>0$ for which
$$|\p_{i}S_j(X)| \leq C_1/(2r+1)^d \ \mbox{for all}\ i,j\ \mbox{and all}\ X \ \mbox{with}\ {\rm osc}(X)\leq K+1\, .$$
This estimate in turn implies for every $X$ with ${\rm osc}(X)\leq K+1$ that
$$|R_i(X)| \leq \sum_{||j-i||\leq r} |\p_i S_j(X)| \leq C_1 \, .$$ 
This proves that $||R(X)||_{\infty}<\infty$ if ${\rm osc}(X)$ is finite. More precisely, because ${\rm osc}(x)\leq K$ and ${\rm osc}(X)\leq {\rm osc }(x)+{\rm osc}(X-x) \leq  {\rm osc }(x)+ 2||X-x||_{\infty}$, we find in particular that $||F(X, \varepsilon)||_{\infty}<\infty$ whenever $||X-x||_{\infty}< \frac{1}{2}$. In other words, each one of the operators $F(\cdot, \varepsilon)$ maps the collection $$B_{1/2}(x)= \{X\in\R^{\Z^d}\ | \ ||X-x||_{\infty} < 1/2 \} \ \mbox{inside}\ l_{\infty}=\{F\in\R^{\Z^d}\ | \ ||F||_{\infty}  < \infty \}\, .$$

%More importantly, because ${\rm osc}(x)\leq K$ and ${\rm osc}(X)\leq {\rm osc }(x)+{\rm osc}(X-x)$, it implies that 
\noindent The next remark is that the Fr\'echet derivative $D_XF(x,0): l_{\infty} \to l_{\infty}$ at $x$ is given by
\[
\left( D_XF(x,0)\cdot v \right)_i:=\left.\frac{d}{dt}\right|_{t=0}\!\! F(x+tv,0)_i = V''(x_i) \cdot v_i\, .
\] 
Because $V$ is a Morse function, $|V''(x_i)| > c$ for some constant $c>0$ and for all $i\in \Z^d$, and it thus follows that $D_XF(x,0): l_{\infty}\to l_{\infty}$ has a bounded inverse.

This is what motivates us to define the quasi-Newton operator $K_{\varepsilon, x}$ by 
$$K_{\varepsilon, x}(X):=X-D_XF(x,0)^{-1}\cdot F(X,\varepsilon), \ \mbox{that is}\ K_{\varepsilon, x}(X)_i:= X_i - \frac{V'(X_i) + \varepsilon R_i(X)}{V''(x_i)}\ .$$
As observed above, $K_{\varepsilon, x}$ maps $B_{1/2}(x)$ into $\{X\, |\, ||X-x||_{\infty}<\infty\} $ and it is clear that $K_{\varepsilon, x}(X)=X$ if and only if $F(X,\varepsilon)=0$. Restricted to an appropriately chosen small ball $B_{\delta_0}(x)$ around $x$, the operator $K_{\varepsilon, x}$ moreover acts as a very strong contraction that sends $B_{\delta_0}(x)$ into itself. 

This can for example be seen from the following standard argument. Let us choose any desired contraction constant $0<k<1$ and, accordingly, a $0<\delta_0<\frac{1}{2}$ with the property that $|V''(X_i)-V''(x_i)|\leq \frac{kc}{2}$ uniformly for $X\in B_{\delta_0}(x)$. Such $\delta_0$ exists because $V''$ is assumed continuous and because $V$ only has finitely many geometrically distinct stationary points. For example, when $V''$ is Lipschitz continuous with Lipschitz constant $L$, then it suffices to choose $\delta_0=\frac{kc}{2L}$. For later reference, let us remark that it holds automatically that the intervals
\begin{align}
\label{disjointintervals}
 (x_j-\delta_0, x_j+\delta_0) \ \mbox{and}\ (x_k-\delta_0, x_k+\delta_0) \ \mbox{are disjoint if} \ x_j\neq x_k\ \mbox{are critical points of}\ V\, .
\end{align}

 We now have, for $X,Y\in B_{\delta_0}(x)$, that
\begin{align}\label{contraction}
|K_{\varepsilon, x}(X)_i-K_{\varepsilon, x}(Y)_i| \leq \left|X_i-Y_i -\frac{V'(X_i)-V'(Y_i)}{V''(x_i)}\right| + \left|\frac{\varepsilon R_i(X) -\varepsilon R_i(Y)}{V''(x_i)}\right|\, .
\end{align}
To estimate the first term in (\ref{contraction}) we write
$$X_i-Y_i - \frac{V'(X_i)-V'(Y_i)}{V''(x_i)}= \frac{X_i-Y_i}{V''(x_i)}\left(\int_0^1\!V''(x_i)-V''(tX_i+(1-t)Y_i)\ \!dt\right)  \, .$$  
This shows that the first term in (\ref{contraction}) is bounded from above by $\frac{k}{2}||X-Y||_{\infty}$.

To  estimate the second term in (\ref{contraction}) we argue as above to conclude that there exists a constant $C_2>0$ for which
$$|\p_{i,k}S_j(X)| \leq C_2/(2r+1)^{2d} \ \mbox{for all}\ i,j,k\ \mbox{and all}\ X \ \mbox{with}\ {\rm osc}(X)\leq K+1\, .$$
In turn, this implies for all $X, Y$ with ${\rm osc}(X), {\rm osc}(Y)\leq K+1$ that
\begin{align}\nonumber 
& |R_i(X) - R_i(Y)| = \left| \int_0^1  \frac{d}{dt} R_i(t X +(1-t)Y)  dt \right|\\ \nonumber & =   \left| \int_0^1  \sum_{||j-i||\leq r} \frac{d}{dt} \p_iS_j(t X +(1-t)Y)   dt \right|   \\ \nonumber & \leq  \sum_{||k-j||\leq r}\sum_{||j-i||\leq r} \int_0^1 |\p_{i, k}S_j(t X + (1-t)Y)|\, dt \cdot |X_k-Y_k| \\ \nonumber &  \leq C_2 ||X-Y||_{\infty}\, .
\end{align}
Thus, we have proved that $R$ is Lipschitz continuous: $$||R(X)-R(Y)||_{\infty} \leq C_2 ||X-Y||_{\infty}\ \mbox{for} \ X,Y\in B_{1/2}(x)\, .$$
Summarizing, we found that  
 $$||K_{\varepsilon, x}(X)-K_{\varepsilon, x}(Y)||_{\infty}\leq \left(\frac{k}{2}+\frac{\varepsilon \, C_2}{c}\right) ||X-Y||_{\infty}\ \mbox{for} \ X, Y \in B_{\delta_0}(x)\, .$$
To investigate whether $K_{\varepsilon, x}$ sends $B_{\delta_0}(x)$ to itself, we observe that when $||X-x||_{\infty}< \delta_0$, then
$$|K_{\varepsilon, x}(X)_i-x_i|\leq |K_{\varepsilon, x}(X)_i-K_{\varepsilon, x}(x)_i|+ |K_{\varepsilon, x}(x)_i - x_i |< \left(\frac{k}{2}+\frac{\varepsilon \, C_2}{c}\right) \delta_0+ \frac{\varepsilon \, C_1}{c}\ ,$$
where the final estimate holds because $V'(x_i)=0$ and $|R_i(x)| \leq C_1$.

From the estimates above it follows that it suffices to choose $0\leq \varepsilon\leq \varepsilon_0 := \min\left\{\frac{kc}{2C_2}, \frac{(1-k)\delta_0c}{C_1}\right\}$ to make sure that both $||K_{\varepsilon, x}(X)-K_{\varepsilon, x}(Y)||_{\infty}\leq k||X-Y||_{\infty}$ and $K_{\varepsilon, x}(B_{\delta_0}(x)) \subset B_{\delta_0}(x) $. Clearly, $\lim_{\delta_0\searrow 0}\varepsilon_0=0$.
\end{proof}

\noindent When $\Gamma$ is any collection of solutions to (\ref{antilimit}) with the property that ${\rm osc}(x)\leq K<\infty$ for all $x\in \Gamma$, then we denote its continuation by
$$\Gamma_{\varepsilon}:=\{x(\varepsilon)\, |\, x\in\Gamma\}\ \mbox{for}\ 0\leq \varepsilon\leq \varepsilon_0\, .$$
Theorem \ref{continuous} below says that the topology of $\Gamma_{\varepsilon}$ is the same as that of $\Gamma$. The proof of Theorem \ref{continuous} is a byproduct of the proof of Theorem \ref{continuation}. The statement of Theorem \ref{continuous} is also contained in \cite{MackayMeiss}, but since we do not understand the proof given in the latter paper, we provide one ourselves. It does not appear as trivial to us as is claimed in \cite{MackayMeiss}.
\begin{theorem}\label{continuous} Let $\Gamma$ be a collection of solutions to (\ref{antilimit}) with ${\rm osc}(x)\leq K<\infty$ for all $x\in\Gamma$. Let $\delta_0, \varepsilon_0>0$ be as in Theorem \ref{continuation}. For $0\leq \varepsilon\leq \varepsilon_0$ and $x\in \Gamma$, we denote by $x(\varepsilon)\in B_{\delta_0}(x)$ its unique continuation that solves (\ref{FKepsilon2}) and by 
$\Gamma_{\varepsilon}:=\{x(\varepsilon)\, |\, x\in\Gamma\}$.
Then the map $$\Phi_{\varepsilon}: \Gamma\to\Gamma_{\varepsilon}\, ,\ x\mapsto x(\varepsilon)$$
is a homeomorphism in the topology of pointwise convergence.
\end{theorem}
\begin{proof}
$\Phi_{\varepsilon}$ is surjective by definition. Injectivity follows because when $x^1(\varepsilon)=x^2(\varepsilon)$, then $B_{\delta_0}(x^1)\cap B_{\delta_0}(x^2)$ is nonempty. This implies that $x^1=x^2$ because both solve (\ref{antilimit}).

To prove that $\Phi_{\varepsilon}$ has a continuous inverse, assume that $x^1, x^2, \dots, x \in \Gamma$ and that $x^n(\varepsilon)\to x(\varepsilon)$ pointwise. This implies that 
$$\lim_{n\to\infty}|x_i^n -x(\varepsilon)_i|\leq \lim_{n\to\infty}|x^n_i -x^n(\varepsilon)_i| +  |x^n(\varepsilon)_i-x(\varepsilon)_i| \leq \delta_0\ .$$ 
But this means that $x(\varepsilon)\in B_{\delta_0}(\lim_{n\to\infty}x^n)$, i.e. $\lim_{n\to\infty}x^n=x$ by uniqueness.

Finally, to prove that $\Phi_{\varepsilon}$ is continuous with respect to pointwise convergence, let $M_1 < M_2$ be integers and choose $x, y\in\Gamma$ with the property that $x|_{B_{M_2}(0)}=y|_{B_{M_2}(0)}$, where $B_M(0)=\{i\in\Z^d\ | \ ||i||\leq M\}$. We will show that it can be arranged that $||x(\varepsilon) - y(\varepsilon)|_{B_{M_1}(0)}||_{\infty}$ is as small as we want if we choose $M_2$ sufficiently much larger than $M_1$. %More precisely, we claim that $||x(\varepsilon)-y(\varepsilon)|_{B_{M_1}}||_{\infty} \leq 2^{M_1-M_2+2}$. 

To prove this, recall that $x(\varepsilon)$ and $y(\varepsilon)$ are obtained as fixed points of the quasi-Newton contraction operators $K_{\varepsilon, x}, K_{\varepsilon, y}$ defined by
$$K_{\varepsilon, x}(X)_i:= X_i - \frac{V'(X_i) + \varepsilon R_i(X)}{V''(x_i)}\ \mbox{and} \ K_{\varepsilon, y}(Y)_i:= Y_i - \frac{V'(Y_i) + \varepsilon R_i(Y)}{V''(y_i)}\ ,$$
on the balls $\{X\in\R^{\Z^d}\ | \ ||X-x||_{\infty}< \delta_0\}$ and $\{Y\in\R^{\Z^d}\ | \ ||Y-y||_{\infty}< \delta_0\}$ respectively.

Both $K_{\varepsilon, x}$ and $K_{\varepsilon, y}$ have very small contraction constants, say bounded by $\frac{1}{2}$. Therefore, 
$$||x(\varepsilon)-K_{\varepsilon, x}^m(x)||_{\infty}\leq \delta_0 2^{-m}\ \mbox{and} \ ||y(\varepsilon)-K_{\varepsilon, y}^m(y)||_{\infty}\leq \delta_0 2^{-m} \, .$$
At the same time, because $x_i=y_i$ for $i\in B_{M_2}(0)$ and since $X\mapsto R_i(X)$ and $Y\mapsto R_i(Y)$ are of finite range $r$, it follows that $K^{m}_{\varepsilon, x}(x)_i = x_i = y_i = K_{\varepsilon, y}^{m}(y)_i$ for all $i\in B_{M_2-mr}(0)$. In particular, $K^{m}_{\varepsilon, x}(x)_i= K_{\varepsilon, y}^{m}(y)_i$ for all $i\in B_{M_1}(0)$ so long as $M_2-m r\geq M_1$, that is if $M_2\geq M_1+ m r$. For such $m$ we conclude that 
\begin{align}\nonumber
 ||x(\varepsilon)-y(\varepsilon)|_{B_{M_1}(0)}||_{\infty} & \leq  \\ ||x(\varepsilon)-K_{\varepsilon, x}^{m}(x)|| + ||K^{m}_{\varepsilon, x}(x) - K_{\varepsilon, y}^{m}(y)|_{B_{M_1(0)}}||  & + ||y(\varepsilon)-K_{\varepsilon, y}^{m}(y) || \leq 2 \delta_0 2^{-m}\ .\nonumber
\end{align}
This implies that $\Phi_{\varepsilon}$ is continuous as follows. Assume that $x^1, x^2, \ldots, x\in \Gamma$ and $x^n\to x$ pointwise. Moreover, let $\varepsilon>0$ be a positive number and $M_1$ an integer. Choose an $m$ so that $2\delta_0 2^{-m}<\varepsilon$ and a corresponding $M_2$ so that $M_2\geq M_1+m r$. Because the critical points of $V$ are discrete, there exists an integer $N$ so that for all $n>N$ it holds that $x^n_i=x_i$ for all $i\in B_{M_2}(0)$. For such $n$ it then holds that $||x^n(\varepsilon)-x(\varepsilon)|_{B_{M_1}(0)}|| <\varepsilon$. Thus, $\lim_{n\to\infty} \Phi_{\varepsilon}(x^n)=\lim_{n\to\infty} x^{n}(\varepsilon) = x(\varepsilon)=\Phi_{\varepsilon}(x)$ pointwise if $\lim_{n\to\infty}x^n=x$ pointwise.
% and be fixed and let us remark that $x+y\in\R^{\Z^d}$ is a solution to (\ref{FKepsilon}) precisely when $G(y):=F(x+y)=0$, with $F$ as defined in (\ref{defF}). We will view $G$ as an operator on the space $\mathbb{X}:=\{y:\Z^d\to\R\ | \ ||y||<\infty\ \}$, with $||y||_{\mathbb{X}}$ defined as 
%$$||y||_{\mathbb{X}}:=\sum_{i\in\Z^d} 2^{-||i||} |y_i|\ . $$
%It clear that $G(0)=0$. {\it It is probably not true that $G$ is continuously differentiable}. Its Fr\'echet derivative $D_yG(0):\mathbb{X}\to\mathbb{X}$ is given by 
%$$(D_yG(0)\cdot v)_i = ( \varepsilon \Delta v)_i- V''(x_i)v_i\ .$$
%We remark that the multiplication operator $V''(x): \mathbb{X}\to\mathbb{X}$ defined by $(V''(x)v)_i:=V''(x_i)v_i$ is clearly invertible with bounded inverse. Also, the norm of the discrete Laplace operator $\varepsilon \Delta:\mathbb{X}\to\mathbb{X}$ is bounded, by $2d\varepsilon$. These two bounds imply that $D_yG(0)$ is invertible, with its inverse given by the Neumann series 
%$$D_yG(0)^{-1}=\left(1+\varepsilon \left(V''(x)^{-1}\circ\Delta\right) +\varepsilon^2\ldots \right)\circ \left(-V''(x)\right)^{-1}\ .$$
%This implies that if $||y-x||_{\mathbb{X}}$ is small, then also $||y(\varepsilon)-x(\varepsilon)||_{\mathbb{X}}$ is small: $||y(\varepsilon)-x(\varepsilon)||\leq ||x-y||+||x(\varepsilon)-x||+||y(\varepsilon)-y||$.
\end{proof}

%\noindent As a particular consequence of Theorem \ref{continuous} it holds for $0\leq \varepsilon <\varepsilon_0$ that the collection $\Gamma_{\varepsilon}$ of solutions to (\ref{FKepsilon}) is totally disconnected, simply because the collection $\Gamma$ of solutions to (\ref{antilimit}) is: indeed, there is a constant $\gamma>0$ so that when $x, y\in \Gamma$ are not equal, then $||x-y||_{\infty}>\gamma$. In fact, this estimate holds for all nonidentical solutions to (\ref{antilimit}), because $V$ is Morse.

\section{Birkhoff configurations}\label{birkhoffsection}
We shall be interested in solutions to (\ref{FKepsilon2}) with the so-called {\it Birkhoff property}, because these are the solutions that generate the laminations alluded to in the introduction. The goal of this section is to characterize laminations as explicitly as possible. Most results in this section are well-known, although not usually proved in full detail in the literature. Also, we believe that Theorems \ref{prevaguetheorem} and \ref{vaguetheorem} have not been formulated explicitly before.

In order to define Birkhoff configurations, we introduce a partial ordering on the space of configurations as follows:
\begin{definition}
For $x, y:\Z^d\to\R$ with $x\neq y$ we write $x<y$ if $x_i\leq y_i$ for all $i\in \Z^d$. Similarly for $x>y$. We say that $x$ and $y$ are {\it ordered} if $x<y$, $x=y$ or $x>y$. 

We write $x\ll y$ if $x_i< y_i$ for all $i\in \Z^d$. Similarly for $x\gg y$. We say that $x$ and $y$ are {\it strictly ordered} if $x\ll y$, $x=y$ or $x\gg y$. 
\end{definition}
Recall the translation maps $\tau_{k,l}:\R^{\Z^d}\to \R^{\Z^d}$ introduced in Definition \ref{taudef}.
%\begin{definition}
%When $x:\Z^d\to\R$ is a configuration, $k\in\Z^d$ and $l\in\Z$, then the translated configuration $\tau_{k,l}x:\Z^d\to\R$ is defined by
%$$(\tau_{k,l}x)_i=x_{i+k}+l\ .$$
%\end{definition}
Now we can define Birkhoff configurations:
\begin{definition}
We say that a configuration $x:\Z^d\to\R$ is {\it Birkhoff} if the collection 
$$\{\tau_{k,l}x\ |\  k\in \Z^d, l\in\Z\}\ \mbox{is ordered}.$$ 
%Equivalently, this means that if $\tau_{k,l}x\neq x$, then either $(\tau_{k,l}x)_i\leq x_i$ for all $i\in\Z^d$ or $(\tau_{k,l}x)_i\geq x_i$ for all $i\in\Z^d$.
\end{definition}
%First of all, let us fix a nonempty and translation-invariant collection $\Sigma\subset \R$ of uniformly nondegenerate critical points of $V$. That is, we require that $V'(y)=0$ for all $y\in \Sigma$, that there is a constant $c>0$ so that $|V''(y)|\geq c$ for all $y\in\Sigma$ and that $y\in \Sigma$ if and only if $y+ 1 \in\Sigma$. We remark that such a $\Sigma$ exists only if $V$ possesses at least one nondegenerate critical point. At the same time, as soon as $V$ is a Morse-function, the collection of all its critical points satisfies our requirements. Our assumptions imply that $\Sigma$ is discrete.
Every Birkhoff configuration has a rotation vector, as was already known to Poincar\'e:
 \begin{lemma}
When $x$ is a Birkhoff configuration, then there is a unique $\omega\in\R^d$ such that 
$$|x_i-(x_0+\omega\cdot i)|\leq 1 \, .$$
In particular, every Birkhoff configuration is of bounded oscillation (namely at most $r||\omega||+2$) and has a rotation vector: $$\lim_{n\to\pm\infty}\frac{x_{ni}}{n}=\omega\cdot i\, .$$
\end{lemma}
For a proof of this lemma, see \cite{MramorRink1}. The rotation vector of a Birkhoff configuration $x$ decides to a large extent wether $\tau_{k,l}x> x$ or $\tau_{k,l}x < x$. We will make such extensive use of the result below that we decided to include it in our exposition. Its simple proof is given in Appendix \ref{proofappendix}.
\begin{proposition}\label{numbertheory}
Let $\omega\in \R^d$ and let $x$ be a Birkhoff configuration with rotation vector $\omega$. If $\omega\cdot k+l>0$, then $\tau_{k.l}x > x$ and if $\omega\cdot k+l<0$, then $\tau_{k.l}x < x$. 
\end{proposition}

\subsection{Hull functions}
%Corollary \ref{clas1} can be seen as a classification result for minimal sets of Birkhoff configurations. 
Important tools for the study of a Birkhoff configuration are its so-called hull functions. %This is the content of Theorem \ref{clas2} below.

\begin{definition}\label{defhull}
Let $x:\Z^d\to\R$ be a Birkhoff configuration of rotation vector $\omega\in\R^d\backslash\Q^d$. We define the {\it lower hull function} $\phi_x^-:\R\to\R$ and the {\it upper hull function} $\phi_x^+:\R\to\R$ of $x$ by
\begin{align}
\phi_x^-(s):= & \lim_{n\to\infty} x_{k_n}+l_n\ \mbox{for any sequence} \ (k_n, l_n) \ \mbox{with} \ \omega\cdot k_n+l_n \nearrow s \ \mbox{as} \ n\to\infty \ . \\
\phi_x^+(s):= & \lim_{n\to\infty} x_{k_n}+l_n\ \mbox{for any sequence} \ (k_n, l_n) \ \mbox{with} \ \omega\cdot k_n+l_n \searrow s \ \mbox{as} \ n\to\infty \ .
\end{align}
\end{definition}
To avoid confusion, let us remark that the requirement in Definition \ref{defhull} that $\omega\cdot k_n+l_n\nearrow s$ as $n\to\infty$ means that $\lim_{n\to\infty} \omega\cdot k_n + l_n=s$ and $\omega\cdot k_n+l_n<s$ for all $n\in \N$. Similarly, by $\omega\cdot k_n+l_n\searrow s$ as $n\to\infty$ we mean that $\lim_{n\to\infty} \omega\cdot k_n + l_n=s$ and $\omega\cdot k_n+l_n>s$ for all $n\in \N$.

We summarize some well-known properties of the hull functions in the following proposition. Its proof will be given in Appendix \ref{proofappendix} and is fully based on Proposition \ref{numbertheory}. Although this proof is elementary, it is a bit delicate at some points nonetheless.
\begin{proposition}\label{propertieshull}
Let $x:\Z^d\to\R$ be a Birkhoff configuration of rotation vector $\omega\in\R^d\backslash\Q^d$.
\begin{itemize}
\item[1)] $\phi_x^-$ and $\phi_x^+$ are well-defined.
\item[2)] $\phi_x^-\leq \phi_x^+$ and $\phi_x^-(\omega\cdot i)\leq x_i \leq \phi_x^+(\omega\cdot i)$ for all $i\in \Z^d$.
\item[3)] $\phi_x^-$ and $\phi_x^+$ are nondecreasing.
\item[4)] $\phi_x^-(s+1)=\phi_x^-(s)+1$ and $\phi_x^+(s+1)=\phi^+_x(s)+1$ for all $s\in \R$. 
\item[5)] $\phi^-_{\tau_{k,l}x}(s)=\phi^-_{x}(s+\omega\cdot k+l)$ and $\phi^+_{\tau_{k,l}x}(s)=\phi^+_{x}(s+\omega\cdot k+l)$.
\item[6)] $\phi_x^-$ is left-continuous and $\phi_x^+$ is right-continuous. 
\item[7)] If $s_n\nearrow s$, then $\phi^+_x(s_n)\to\phi_x^-(s)$ and if $s_n\searrow s$, then $\phi^-_x(s_n)\to\phi_x^+(s)$.
\item[8)] $\phi_x^-(s)=\phi_x^+(s)$ if and only if $\phi_x^-$ is continuous at $s$ if and only if $\phi_x^+$ is continuous at $s$. 
%\item[9)] $\phi_x^-(y)=\phi_x^+(y)$ if and only if $\phi_x^-(y+\omega\cdot k+l)=\phi_x^+(y+\omega\cdot k+l)$.
\item[9)] $\phi_x^-=\phi_x^+$ Lebesgue almost everywhere.
\end{itemize}
%Here, $\Phi_x$ is the upper hull function of $x$.
\end{proposition}
%The hull functions of $x$ can be used to characterize $\Gamma(x)$ as follows.
Hull functions are in particular convenient for the study of recurrent Birkhoff configurations:
\begin{definition}\label{recurrentdef}
We say that a Birkhoff configuration $x:\Z^d\to\R$ of rotation vector $\omega\in\R^d\backslash \Q^d$ is {\it recurrent} if there exists a sequence $(k_n,l_n)\in \Z^d\times \Z$ with $\omega\cdot k_n + l_n \neq 0$ such that $$x= \lim_{n\to\infty} \tau_{k_n,l_n}x\ \mbox{pointwise}.$$ 
\end{definition}
It follows from Proposition \ref{numbertheory} that for the sequence $(k_n, l_n)$ of Definition \ref{recurrentdef} it must hold that $\lim_{n\to\infty} \omega\cdot k_n +l_n=0$ and therefore that $||k_n||\to \infty$ and $|l_n|\to\infty$ as $n\to \infty$. By passing to a subsequence, one can assume that either $\omega\cdot k_n + l_n \nearrow 0$ or that  $\omega\cdot k_n + l_n \searrow 0$. In the first case, it holds that $\omega\cdot (k_n+i) + l_n \nearrow \omega\cdot i$ and hence $x_i = \lim_{n\to \infty} (\tau_{k_n, l_n}x)_i = \lim_{n\to \infty}  x_{k_n+i}+l_n = \phi_x^-(\omega\cdot i)$. Similarly, in the second case $x_i =\phi^{+}_x(\omega\cdot i)$. Thus, a recurrent Birkhoff configuration can be reconstructed by ``sampling'' one of its hull functions. Conversely, any configuration $X$ of the form $X_i:=\phi_x^{\pm}(s+\omega\cdot i)$ is recurrent. This is one of the consequences of the following theorem, of which we defer the proof to Appendix \ref{proofappendix}.

%$${\rm either} \ x_i=\phi^-_x(\omega\cdot i)\ \mbox{for all}\ i\in\Z^d\ \mbox{or}\ x_i=\phi^+_x(\omega\cdot i)\ \mbox{for all}\ i\in \Z^d\, .$$
\begin{theorem} \label{xplusminuslemma} Let $x$ be a Birkhoff configuration of irrational rotation vector $\omega\in \R^d\backslash \Q^d$ and let $\phi_x^{\pm}$ denote its hull functions.  Let us define 
$$\Gamma(x) := \{x^{\pm}(s)\ | \ s\in \R\}\subset \R^{\Z^d} \ \mbox{where}\ x^-(s)_i:= \phi_x^-(s+\omega\cdot i) \ \mbox{and} \ x^+(s)_i:=\phi^+_{x}(s+\omega\cdot i)\, . $$ 
Then $\Gamma(x)$ is ordered, translation-invariant, closed under pointwise convergence and minimal. In particular, every element of $\Gamma(x)$ is recurrent.
\end{theorem}
Here, minimality means that $\Gamma(x)$ does not have any nonempty proper subset that is also closed and shift-invariant. We think of $\Gamma(x)$ as the minimal set generated by $x$. As explained above, it holds that $x\in \Gamma(x)$ if and only if $x$ is recurrent, and it follows from property {\it 2)} of Proposition \ref{propertieshull} that the collection $\{x\}\cup \Gamma(x)$ is ordered even if $x\notin \Gamma(x)$.

The following well-known result specifies the topological structure of $\Gamma(x)$. We recall that a topological space $\mathcal{C}$ is called a {\it Cantor set} if it is closed, perfect and totally disconnected. ``Perfect'' means that every element $c\in \mathcal{C}$ is a limit of points in $\mathcal{C}\backslash \{c\}$. ``Totally disconnected'' means that for any two elements $c_1, c_2\in \mathcal{C}$ one can decompose $\mathcal{C}$ as the disjoint union of closed sets $\mathcal{C}_1$  and $\mathcal{C}_2$ with $c_1\in \mathcal{C}_1$ and $c_2\in \mathcal{C}_2$. 

\begin{theorem}\label{Cantortheorem}
When $x$ is a Birkhoff configuration with rotation vector $\omega\in \R^d\backslash \Q^d$, then $\Gamma(x)$ is either a topologically connected or a Cantor subset of $\R^{\Z^d}$.
\end{theorem}
Using the hull function, we provide a proof of this fact in Appendix \ref{proofappendix}. 

When $\Gamma(x)$ is connected, then we say that it forms a {\it foliation}. In case it is a Cantor set, one says that it forms a {\it lamination}. This terminology is probably due to Moser \cite{moser86}, \cite{moser89}.

We attribute the following result to Bangert \cite{bangert87}. It establishes the one-to-one correspondence between hull functions and minimal sets.

%If $x$ and $X$ are two Birkhoff configurations of rotation vector $\omega\in \R^d\backslash \Q^d$ and if there exists an $s\in \R$ so that $\phi^-_x(\cdot)=\phi_X^-(\cdot + s)$ (that is the hull functions of $x$ and $X$ are the same up to translation) then it is clear that $\Gamma(x) = \Gamma(X)$. The opposite is also true:

%As a corollary of Lemma \ref{xplusminuslemma} we obtain that the minimal subsets of two different Birkhoff configurations can be distinguished from their respective hull functions:
\begin{theorem}\label{clas2}
Let $x$ and $y$ be two Birkhoff configurations with rotation vector $\omega\in\R^d\backslash\Q^d$. Then the following are equivalent:
\begin{itemize}
\item[{\it 1)}] $\Gamma(x)=\Gamma(y)$. 
\item[{\it 2)}] $\Gamma(x)\cup\Gamma(y)$ is ordered.
\item[{\it 3)}] There is an $s\in\R$ for which $\phi_x^-(\cdot) = \phi_y^-(\cdot + s)$.
\item[{\it 4)}] There is an $s\in\R$ for which $\phi_x^+(\cdot) = \phi_y^+(\cdot + s)$.
\end{itemize}
%\item There is an $s\in\R$ for which $\phi^+_x(\cdot) = \phi_X^+(\cdot + s)$.
\end{theorem} 
%Remark that, obviously, $\phi^-_x(\cdot)=\phi^-_X(\cdot +s)$ if and only if $\phi^+_x(\cdot) = \phi_X^+(\cdot + s)$.
%For the proof of Theorem \ref{clas2}, we need following simple crossing lemma that we formulate separately:
\begin{proof} 
First of all, in view of property {\it 7)} of Proposition \ref{propertieshull}, it is clear that {\it 3)} and {\it 4)} are equivalent. Let us now define 
$$s^-:=\max \{s\in\R\ | \ \phi_x^-(\cdot) \geq \phi_y^-(\cdot+s)\}\ \mbox{and}\ s^+:= \min \{s\in\R\ | \ \phi_x^-(\cdot)\leq \phi_y^-(\cdot+s) \}\, .$$
We claim that these quantities exist. Indeed, if $\phi_y^-(\cdot+s_n)\leq \phi_x^-(\cdot)$ for a sequence $s_n\nearrow s_{\infty}$, then $\phi_y^-(\cdot+s_{\infty})\leq \phi_x^-(\cdot)$ because $\phi_y^-$ is left-continuous. Thus, the maximum in the definition of $s^-$ is attained. To prove that $s^+$ exists, one remarks that $s^+=\min\{s\in\R\ | \ \phi_x^-(\cdot-s)\leq\phi_y^-(\cdot)\}$. Thus, if $\phi_x^-(\cdot-s_n)\leq \phi_y^-(\cdot)$ for a sequence $s_n\searrow s_{\infty}$, then $\phi_x^-(\cdot-s_{\infty})\leq \phi_y^-(\cdot)$.

Now we consider two cases. First of all, it may happen that $s^-=s^+$. Then $\phi_x^-(\cdot)=\phi_y^-(\cdot+s^-)$, that is $\phi_x^-$ is a translate of $\phi_y^-$. Then also $\phi_x^{+}(\cdot)=\phi_y^+(\cdot+s^-)$ and therefore the sets $\{x^{\pm}(s)\, |\, s\in\R\}$ and $\{y^{\pm}(s)\, |\, s\in\R\}$ coincide, that is $\Gamma(x)=\Gamma(y)$. %Similarly, if $\phi_x^-(\cdot)=\phi_X^-(\cdot+s)$ for some $s\in \R$, then also $\Gamma(x)=\Gamma(X)$.

If $s^-\neq s^+$, then it must hold that $s^-<s^+$ because $\phi_x^-$ and $\phi_y^-$ are nondecreasing. For any $s^-<s<s^+$, there then are $s_1, s_2$ so that $\phi_x^-(s_1)<\phi_y^-(s_1+s)$ and $\phi_x^-(s_2)>\phi_y^-(s_2+s)$. By periodicity of the hull functions, we can choose $s_1<s_2$.

%According to Proposition \ref{crossinglemma} there then are $s\in \R$ and $y_1<y_2$ so that $\phi_x^-(y_1) < \phi_X^-(y_1+s)$ and $\phi_x^-(y_2)>\phi_X^-(y_2+s)$. 

We now claim that $x^-(s_1)$ and $y^-(s_1+s)$ cross. Indeed, it holds that $x^-(s_1)_0=\phi_x^-(s_1)<\phi_y^-(s_1+s)=y^-(s_1+s)_0$. At the same time, choosing $k_n,l_n$ such that $s_1+ \omega\cdot k_n+l_n\nearrow s_2$, we have that $\lim_{n\to\infty} x^-(s_1)_{k_n}+l_n = \phi_x^-(s_2)>\phi_y^-(s_2+s)=\lim_{n\to\infty}y^-(s_1+s)_{k_n}+l_n$. Thus, there is an $n$ for which $x^-(s_1)_{k_n}>y^-(s_1+s)_{k_n}$. Because $x^-(s_1)\in\Gamma(x)$ and $y^-(s_1+s)\in\Gamma(y)$ this means that $\Gamma(x) \cup \Gamma(y)$ is not ordered and therefore certainly $\Gamma(x)\neq \Gamma(y)$. This proves the theorem.
\end{proof}
\noindent
Let us finish this section with a result that relates pointwise convergence of recurrent Birkhoff configurations to the convergence of their hull functions:
%\begin{lemma}\label{crossinglemma}
%Let $\phi_x^-$ and $\phi_X^-$ be lower hull functions and assume that there is no $s\in\R$ such that $\phi_x^-(\cdot)=\phi_X^-(\cdot+s)$. Then there are $s\in\R$ and $y_1< y_2$ so that 
%$$\phi_x^-(y_1) < \phi_X^-(y_1+s)\ \mbox{and}\ \phi_x^-(y_2)>\phi_X^-(y_2+s) \ .$$
%\end{lemma}
%\begin{proof} 
%\end{proof}
 
\begin{theorem}\label{prevaguetheorem}
Let $x^1, x^2, \ldots, x$ be recurrent Birkhoff configurations of rotation vector $\omega\in\R^d\backslash \Q^d$ and let $\phi^{\pm}_{x^1}, \phi^{\pm}_{x^2}, \ldots, \phi^{\pm}_x$ denote their respective hull functions. When $\lim_{n\to\infty} x^n=x$, then $\lim_{n\to\infty}\phi_{x^n}^{\pm} = \phi_x^{\pm}$ almost everywhere.
\end{theorem}
\begin{proof}
Because $x^1, x^2, \ldots, x$ are recurrent, it holds that $x_i^n=\phi_{x^n}(\omega\cdot i)$ and $x_i=\phi_x(\omega\cdot i)$ where $\phi_{x^n}$ is one of the two hull functions of $x^n$ and $\phi_{x}$ is one of the two hull functions of $x$. The assumption that $x^n\to x$ pointwise therefore just means that $\phi_{x^n} \to \phi_x$ pointwise on the dense set $\{\omega\cdot i\, |\, i\in \Z^d\} \subset \R$. 

Now we recall that $\phi_{x}$ is continuous almost everywhere and let $t\in\R$ be a point of continuity. Let $\varepsilon>0$ be given and choose a $\delta>0$ so that $\phi_x(t)-\varepsilon < \phi_{x}(s) < \phi_x(t)+\varepsilon$ for all $t-\delta<s<t+\delta$. Next, we choose $t-\delta < s_1<t<s_2<t+\delta$ with the property that $\lim_{n\to\infty} \phi_{x^n}(s_{1,2}) = \phi_x(s_{1,2})$. The inequalities $\phi_{x^n}(s_1)\leq \phi_{x^n}(t)\leq \phi_{x^n}(s_2)$ together with all the above then yield that
$$\phi_x(t)-\varepsilon\leq \phi_x(s_1) = \lim_{n\to\infty} \phi_{x^n}(s_1) \leq \lim_{n\to\infty} \phi_{x^n}(t) \leq  \lim_{n\to\infty} \phi_{x^n}(s_2) = \phi_x(s_2)\leq\phi_x(t)+\varepsilon\, .$$
Since this is true for all $\varepsilon>0$, we find that $\lim_{n\to\infty} \phi_{x^n}(t) = \phi_x(t)$ at any point $t$ of continuity of $\phi_x$, and hence  $\lim_{n\to\infty} \phi_{x^n} = \phi_x$ almost everywhere. The result now follows because $\phi_{x^n}^{\pm}=\phi_{x^n}$ almost everywhere and $\phi_{x}^{\pm}=\phi_{x}$ almost everywhere.
\end{proof}

\begin{example}\label{ex1}
Let $0\leq \sigma_1< \ldots < \sigma_N<1$ and let $p_1, \ldots, p_N\geq 0$ be so that $p_1+\ldots +p_N=1$. 
Then there is precisely one left-continuous hull function $\phi_p^-:\R\to \R$ for which 
$$\lim_{s\nearrow 0}\phi_p^-(s)\leq 0\, ,\ \lim_{s\searrow 0}\phi_p^-(s)> 0\ \mbox{and}\ |(\phi_p^-)^{-1}(\sigma_j)|=p_j\ \mbox{for all}\ 1\leq j\leq N\, .$$ 
In view of Theorem \ref{clas2}, this means that for any $\omega\in \R^d\backslash \Q^d$, the set of laminations of rotation vector $\omega$ and taking values in $\{\sigma_1, \ldots, \sigma_N\}+\Z$, is in one-to-one correspondence with the $(N-1)$-dimensional simplex $\Delta_{N-1}$. 

Moreover, if $p^1, p^2, \ldots, p\in \Delta^{N-1}$ form a sequence for which $\lim_{n\to\infty} p^n=p$, then it is clear that $\lim_{n\to \infty} \phi^-_{p^n}(t)=\phi^-_p(t)$ at every point $t$ at which $\phi_p^-$ is continuous. Because $\phi_p^-$ has at most $N$ geometrically distinct discontinuities, it follows that $\lim_{n\to \infty} \phi_{p^n}^-=\phi^-_p$ almost everywhere.
\end{example}

\subsection{The hull function interpreted as an invariant measure} \label{measuresection}
We now give a statistical interpretation of the hull function and a corresponding characterization of minimal sets. Below, $|\cdot|$ denotes Lebesgue measure on $\R/\Z$, $\sharp$ cardinality and $B_n(0)=\{i\in \Z^d\, | \, ||i||\leq n \, \} \subset \Z^d$ the ball of radius $n$.

\begin{theorem} \label{clas3} Let $x:\Z^d\to \R$ be a recurrent Birkhoff configuration of rotation vector $\omega\in\R^d\backslash \Q^d$ and $\Sigma\subset\R/\Z$ an interval. Then the limit
$$\mu_x(\Sigma):= \lim_{n\to\infty} \frac{\sharp\{i\in B_n(0)\ | \ x_i \!\!\!\! \mod 1 \in \Sigma\}}{\sharp B_n(0)}\ \mbox{exists and equals}\ |\phi_x^{-1}(\Sigma)|\, .$$
Here $\phi_x=\phi_x^{\pm}$ is either one of the two hull functions of $x$ viewed as a map from $\R/\Z$ to $\R/\Z$.
%Moreover, 
 %$$\Gamma(x)=\Gamma(X) \Leftrightarrow \ p_x(\Sigma)=p_X(\Sigma)\ \mbox{for every interval}\ \Sigma\subset\R \ .$$%\lim_{n\to\infty} \frac{\sharp\{i\in B_n\ | \ x_i\in \Sigma+\Z\}}{\sharp B_n}=\lim_{n\to\infty} \frac{\sharp\{i\in B_n\ | \ X_i\in \Sigma+\Z\}}{\sharp B_n}  \ .
%\end{align}
\end{theorem}
\begin{proof}
%[Of Theorem \ref{clas3}] 
Let us recall that when $\omega\in \R^d\backslash \Q^d$, then the action 
$$(i, s\!\!\!\! \mod \! 1)\mapsto s+\omega\cdot i\!\!\!\! \mod \! 1\ \mbox{of}\ \Z^d\ \mbox{on}\ \R/\Z\ $$
is ergodic in the following strong sense: for every $s\in\R$ and every interval $S\subset\R/\Z$, the limit
\begin{align}\label{pxcont1}
\lim_{n\to\infty} \frac{\sharp\{i\in B_n(0)\, | \, s+\omega\cdot i \!\!\!\! \mod 1 \in S\}}{\sharp B_n(0)}\ \mbox{exists and is equal to}\ |S|  \, .
\end{align}
Let us now choose any configuration $x^{\pm}(s) \in \Gamma(x)$ in the minimal set generated by $x$.
Then we can apply (\ref{pxcont1}) to the interval $S=(\phi_x^{\pm})^{-1}(\Sigma)$, where $\Sigma\subset \R/\Z$ is any interval. %Using that $(\phi_x^{-})^{-1}(\Sigma)+\Z= (\phi_x^{-})^{-1}(\Sigma+\Z)$ and 
Recalling that $x^{\pm}(s)_i=\phi_x^{\pm}(s+\omega\cdot i)$, and hence that $x^{\pm}(s)_i \! \mod 1 \in \Sigma$ if and only if $s+\omega\cdot i\in S$, this yields that
\begin{align}\label{pxcont2}
\lim_{n\to\infty} \frac{\sharp\{i\in B_n(0)\ | \ x^{\pm}(s)_i \!\!\!\! \mod 1 \in \Sigma\}}{\sharp B_n(0)}\ \mbox{exists and is equal to}\ |(\phi_x^{\pm})^{-1}(\Sigma)|  \, .
\end{align}
Because $\phi_x^-=\phi_x^+$ almost everywhere, $|(\phi_x^-)^{-1}(\Sigma)| = |(\phi_x^+)^{-1}(\Sigma)|$. 
This finishes the proof. 
%\begin{align}\label{pxcont3}
%p_x(\Sigma)=|(\phi_x^{\pm})^{-1}(\Sigma\cap[0,1))|\, .
%\end{align}
%Applying (\ref{pxcont1}) to $Y=(\phi_x^{+})^{-1}(\Sigma)$ we find in the same way that (\ref{pxcont2}) holds with $\phi_x^-$ replaced by $\phi_x^+$. 
%Now we remark that the right hand side of (\ref{pxcont2}) does not depend on $y$ and moreover, 
%We conclude:
%\begin{align}\label{pxcont3}
%p_x(\Sigma) \ \mbox{is well-defined and equal to}\ |(\phi_x^{\pm})^{-1}(\Sigma \cap [0,1))|\ .
%\end{align}
%Combined with Theorem \ref{clas2} and Proposition \ref{cross2}, (\ref{pxcont3}) implies that $p_x(\Sigma)=p_X(\Sigma)$ for all intervals $\Sigma$ if and only if $\Gamma(x)=\Gamma(X)$. 
%One can actually prove (\ref{pxcont3}) also if $x$ is not recurrent. This works because $x$ is asymptotic to recurrent configurations. More precisely, there exists two recurrent configurations $x^-(y)\leq x \leq x^+(y)$ satisfying
%$$\sum_{i\in \Z^d/H_{\omega}} x^+(y)_i - x^-(y)_i \leq 1, \ \mbox{where}\ H_{\omega}:=\{i\in\Z^d\, | \, \langle \omega, i\rangle =0\}\, .$$
%For a proof of this fact we refer to \cite{RinkMramor1}. 
%Assume now that (\ref{pxcont3}) does not hold for $x$. Then there must exists a sequence of points  
%\begin{align}\label{pxreal}
%p_x(\Sigma)=\lim_{n\to\infty} \frac{\sharp\{i\in B_n\, | \, x_i\in \Sigma+\Z\}}{\sharp B_n}\ .
%\end{align}
%When $\omega\notin \R/\Q$, this implies in particular that 
%Unfortunately the only proof of (\ref{pxreal}) that we know is annoyingly technical. Since (\ref{pxreal}) is also not essential for this paper, we omit its proof.
 \end{proof}
 Being a nondecreasing function, every hull function $\phi_x^{\pm}:\R\to\R$ is clearly Borel measurable, the inverse image of an interval being an interval. Thus, the formula
\begin{align} \label{defmux}
\mu_x(\Sigma) = |\phi_x^{-1}(\Sigma)|
\end{align}
is easily checked to define a Borel probability measure on $\R/\Z$. Theorem \ref{clas3} shows that we can interpret $\mu_x(\Sigma)$ as the ``probability'' that the recurrent Birkhoff configuration $x^{\pm}(s)\in \Gamma(x)$ takes a value in $\Sigma+\Z$. The following not so surprising result states that two such probability measures $\mu_x$ and $\mu_y$ are equal if and only if $x$ and $y$ generate the same minimal set:

\begin{theorem}\label{cross2}
Let $x$ and $y$ be two Birkhoff configurations with rotation vector $\omega\in\R^d\backslash\Q^d$. Then the following are equivalent:
\begin{itemize}
\item[{\it 1)}] $\mu_x=\mu_y$. 
\item[{\it 2)}] There is an $s\in\R$ for which $\phi_x^-(\cdot) = \phi_y^-(\cdot + s)$.
\item[{\it 3)}] There is an $s\in\R$ for which $\phi_x^+(\cdot) = \phi_y^+(\cdot + s)$.
\end{itemize}
\end{theorem}
\begin{proof} 
We prove that {\it 1)} is equivalent to {\it 2)}. Indeed, if $\phi_x^{-}(\cdot) = \phi_y^{-}(\cdot+s)$, then $(\phi_x^{-})^{-1}(\Sigma)=(\phi_y^{-})^{-1}(\Sigma) - s$ and hence $|(\phi_x^-)^{-1}(\Sigma)|=|(\phi_y^-)^{-1}(\Sigma)|$. %  Similarly if $\phi_x^{+}(\cdot)=\phi_X^{+}(\cdot+s)$.  
 Thus $\mu_x=\mu_y$.

In the other direction, let us assume that $\phi_x^-\neq \phi_y^-(\cdot+s)$ for any $s$. As was shown in the proof of Theorem \ref{clas2}, this implies that there are $s_1<s_2$ and an $s$ so that $\phi_x^-(s_1)<\phi_y^-(s_1+s)$ and $\phi_x^-(s_2)>\phi_y^-(s_2+s)$. By left-continuity of $\phi_x^-$ at $s_2$, there thus exists a $s_1<s_3<s_2$ such that $\phi_x^-(s_3)>\phi_y^-(s_2+s)$. This means that $(\phi_x^-)^{-1}[\phi_y^-(s_1+s), \phi_y^-(s_2+s)]\subset[s_1, s_3]$ and $(\phi_y^-)^{-1}[\phi_y^-(s_1+s), \phi_y^-(s_2+s)]=[s_1+s, s_2+s]$. Hence these intervals have a different Lebesgue measure and $\mu_x\neq \mu_y$.
\end{proof}
We remark that formula (\ref{defmux}) implies for any measurable $\Sigma \subset \R/\Z$ that
$$\int_{\R/\Z}{\bf 1}_{\Sigma}d\mu_{x} = \mu_x(\Sigma) = |\phi_x^{-1}(\Sigma)| = \int_{\R/\Z} {\bf 1}_{\phi_x^{-1}(\Sigma)}(s)ds =  \int_{\R/\Z} {\bf 1}_{\Sigma}(\phi_x(s))ds\, .$$
By the usual construction of the Lebesgue integral we therefore obtain that any Lebesgue measurable function $f:\R/\Z\to\R$ is also $\mu_{x}$-measurable and that
\begin{align}\label{changeofvariables}
\int_{\R/\Z}fd\mu_{x} = \int_{\R/\Z} f(\phi_x(s))ds\, .
\end{align}
This observation reveals that $\mu_x$ can be viewed as a Radon probability measure. We recall that a Radon probability measure on $\R/\Z$ is a positive continuous linear functional 
$$\mu: C^0(\R/\Z, \R) \to \R$$ 
on the space of continuous functions on $\R/\Z$ endowed with uniform convergence, and with the property that $\mu(1)=1$. The space of Radon probability measures on $\R/\Z$ will be denoted $M(\R/\Z)$ and it is clear that 
$$\mu_x(f) := \int_{\R/\Z} f d\mu_x = \int_{\R/\Z} \! f(\phi_x(s))ds $$ 
defines a Radon probability measure. We will endow the space $M(\R/\Z)$ with the topology induced by the following weak form of convergence:
\begin{definition}
Let $\mu_1, \mu_2, \ldots, \mu\in M(\R/\Z)$ be Radon probability measures on $\R/\Z$. We say that the $\mu_n$ converge {\it vaguely} to $\mu$ if for every continuous function $f:\R/\Z\to\R$ it holds that
$$\lim_{n\to \infty} \int_{\R/\Z}\!\! f \, d \mu_n =  \int_{\R/\Z}\!\! f \, d\mu\, .$$
\end{definition}
%This means that $\mu_{x_n}$ converges vaguely to $\mu_x$ when the statistical behavior of $\Gamma(x_n)$ converges to that of $\Gamma(x)$.
The above equality simply means that the functionals $\mu_n$ converge pointwise to $\mu$. One also says that they weak-star converge. The topology defined by vague convergence is called the {\it vague topology} on $M(\R/\Z)$. By definition, the vague topology makes $M(\R/\Z)$ a topological Hausdorff space, i.e. vague limits are unique. More on the use of the vague topology in Aubry-Mather theory can be found in \cite{moredenjoy}. 
The final result of this section provides some more intuition for the vague topology. It relates the convergence of hull functions to vague convergence: %We state without proof that the vague topology on $M(\R/\Z)$ is Hausdorff.

\begin{theorem}\label{vaguetheorem}
Let $\phi_1, \phi_2, \ldots$ be hull functions and let $\mu_{1}, \mu_{2},\ldots$ denote the Radon probability measures they induce on $\R/\Z$. When $\lim_{n\to\infty} \phi_n=\phi$ almost everywhere, then $\lim_{n\to\infty}\mu_{n} = \mu$ vaguely.
\end{theorem}
\begin{proof}
When $\lim_{n\to\infty} \phi_{n} = \phi$ almost everywhere and $f:\R/\Z\to\R$ is continuous, then also $\lim_{n\to \infty} f \circ \phi_{x^n} = f \circ \phi_x$ almost everywhere. Because $|(f\circ \phi_n)(s)|\leq \sup_{\sigma\in \R/\Z}|f(\sigma)|$ is uniformly bounded both in $n$ and in $s$, the Dominated Convergence Theorem gives that 
$$\int_{\R/\Z} f d\mu_{n}= \int_{\R/\Z} f(\phi_{n}(s))ds \to \int_{\R/\Z} f(\phi(s))ds = \int_{\R/\Z} f d\mu\, .$$ 
This proves that $\lim_{n\to\infty}\mu_{n} = \mu$ vaguely. 
\end{proof}

\begin{example}\label{ex2}
Let $0\leq \sigma_1< \ldots < \sigma_N<1$ and let $p_1, \ldots, p_N\geq 0$ be so that $p_1+\ldots +p_N=1$. As explained in Example \ref{ex1}, there is exactly one left-continuous hull function $\phi^-_p: \R \to \R$ with
$$\lim_{s\nearrow 0}\phi^-_p(s)\leq 0\, ,\ \lim_{s\searrow 0}\phi_p^-(s)> 0\ \mbox{and}\ |(\phi_p^-)^{-1}(\sigma_j)|=p_j\ \mbox{for all}\ 1\leq j\leq N\, .$$  
The Radon probability measure corresponding to this hull function is simply given by
$$\mu_p(f):=\int_{\R/\Z} f (\phi_p^-(s))ds  =\sum_{j=1}^N p_jf(\sigma_j)\, .$$
This implies in particular that if $p^1, p^2, \ldots, p\in \Delta_{N-1}$ and $\lim_{n\to\infty} p^{n} = p$, then
$$\lim_{n\to \infty} \mu_{p^n}(f)=\lim_{n\to \infty} \sum_{j=1}^N p_j^{n}f(\sigma_j) = \sum_{j=1}^N p_jf(\sigma_j) = \mu_p(f)\, .$$
Hence $\lim_{n\to\infty} \mu_{p^n} = \mu_p$ vaguely. In other words, the map 
$$\Psi: \Delta_{N-1}\to M(\R/\Z)\ \mbox{given by}\ p\mapsto \mu_p$$
is continuous. Because $\Psi$ is clearly injective, $\Delta_{N-1}$ is compact and $M(\R/\Z)$ is Hausdorff, $\Psi$ is actually a homeomorphism onto its image.
\end{example}

\section{Birkhoff solutions near the anti-continuum limit}\label{propertiessection}
The goal of this section and the next is to obtain a complete description of the recurrent Birkhoff solutions to equations (\ref{FKepsilon2}) near the anti-integrable limit. We start with the following simple observation:  
 %shows that it is important to understand how the Birkhoff solutions to the anti-continuum limit (\ref{antilimit}) continue to solutions of (\ref{FKepsilon2}). A first insight is given by the following proposition:

\begin{proposition}\label{Cantorcontinuation}
Let $x$ be a Birkhoff solution to (\ref{antilimit}) of rotation vector $\omega\in \R^d\backslash \Q^d$, let $\varepsilon_0>0$ be as in Theorem \ref{continuation} and let $0\leq \varepsilon\leq\varepsilon_0$. Then 
$$\Gamma_{\varepsilon}(x) := \{ X(\varepsilon)\, |\, X\in \Gamma(x)\, \} \subset\R^{\Z^d}\ \mbox{is a Cantor set and homeomorphic to}\ \Gamma(x) .$$
\end{proposition}
\begin{proof}
Recall that $\Gamma_{\varepsilon}(x)$ is homeomorphic to $\Gamma(x)$ by Theorem \ref{continuous}, so it suffices to prove that $\Gamma(x)$ is a Cantor set. But Theorem \ref{Cantortheorem} applies. Clearly, $\Gamma(x)$ is not topologically connected: it holds that $||X^1-X^2||_{\infty}\geq \gamma$ for some $\gamma>0$ and all $X^1\neq X^2$ in $\Gamma(x)$, because $X^1$ and $X^2$ only take values in a discrete set. Thus, $\Gamma(x)$ is a Cantor set. 
\end{proof}
Proposition \ref{Cantorcontinuation} implies that $\Gamma_{\varepsilon}(x)$ is not a foliation, but of course it does not say that $\Gamma_{\varepsilon}(x)$ is a lamination: it is so if and only if it is ordered. In turn, this is the case if and only if the elements of $\Gamma_{\varepsilon}(x)$ are Birkhoff and then $\Gamma_{\varepsilon}(x)=\Gamma(X(\varepsilon))$ for any element $X\in \Gamma(x)$. Thus, we wonder under what conditions the continuation $x(\varepsilon)$ of a recurrent Birkhoff solution $x$ to (\ref{antilimit}) is still Birkhoff. The answer is given by Theorems \ref{corollaryordering} and \ref{conversethm} below. %Recall the constants $K, \delta_0$ and $\varepsilon_0$ from Theorem \ref{continuation}.

\begin{theorem} \label{corollaryordering}
Let $\omega\in\R^d\backslash \Q^d$, let $K>0$ be so that ${\rm osc}\, (X)\leq K$ for all Birkhoff configurations $X$ of rotation vector $\omega$ and let $\delta_0, \varepsilon_0>0$ be as in Theorem \ref{continuation}. Then there is an $0<\varepsilon_1\leq \varepsilon_0$ so that for all $0\leq \varepsilon\leq \varepsilon_1$ the following is true: 

When $x:\Z^d\to\R$ is a recurrent Birkhoff solution to (\ref{antilimit}) of rotation vector $\omega$ and every $x_i$ is a local minimum of $V$, then $x(\varepsilon)$ is Birkhoff.
%Conversely, if $\Gamma_{\varepsilon}$ is any lamination of solutions for (\ref{FKepsilon}), then there is a Birkhoff configuration $x:\Z^d\to\R$ for which every $x_i$ is a local minimum of $V$ and so that $\Gamma_{\varepsilon}=\Gamma(x(\varepsilon))$.  
\end{theorem}
It turns out that one can choose $\varepsilon_1:=\min\{2c/C_2,\varepsilon_0\}$, with $c$ and $C_2$ as in the proof of Theorem \ref{continuation}.
The proof of Theorem \ref{corollaryordering} is somewhat lengthy and is based on the fact that the solutions $x$ and $x(\varepsilon)$ described in Theorem \ref{corollaryordering} are ``local action minimizers''. We postpone this proof to Section \ref{minimizersection}. It heavily relies on ideas from \cite{bangert87}. %Nevertheless, the following lemma already sheds some light on the main idea of the proof of Theorem \ref{corollaryordering}: if the continuation of a recurrent Birkhoff configuration crosses one of its translates, then it does so infinitely often. For local minimizers, this property will eventually lead to a contradiction. 

The converse of Theorem \ref{corollaryordering} is the following result, that is easier to prove. % that implies that Theorem \ref{ourthm} describes all the Birkhoff solutions of (\ref{Vprimeiszero}):
\begin{theorem}\label{conversethm}
Let $\omega\in\R^d\backslash \Q^d$, let $K>0$ be so that ${\rm osc}\, (X)\leq K$ for all Birkhoff configurations $X$ of rotation vector $\omega$ and let $\delta_0, \varepsilon_0>0$ be as in Theorem \ref{continuation}. Then there is an $0<\varepsilon_1\leq \varepsilon_0$ so that for all $0\leq \varepsilon\leq \varepsilon_1$ the following is true: 

When $x:\Z^d\to \R$ is a recurrent Birkhoff solution to (\ref{antilimit}) of rotation vector $\omega$ and if there is an $i\in \Z^d$ for which $x_i$ is a local maximum of $V$, then $x(\varepsilon)$ is not Birkhoff.
\end{theorem}
\begin{proof}
Let $x$ be as described in the statement of the theorem. Because $x$ is recurrent and only takes discrete values, there is a translate $y=\tau_{k,l}x\neq x$ such that $x_i=y_i$.  It then holds that $x(\varepsilon)_i, y(\varepsilon)_i\in (x_i-\delta_0, x_i+\delta_0) = (y_i-\delta_0, y_i+\delta_0)$, where we may assume that $V''\leq-c<0$ on $(x_i-\delta_0, x_i+\delta_0)$. If we also assume, for example, that $y>x$, then there is a $j$ so that $y_j>x_j$ and hence also $y(\varepsilon)_j>x(\varepsilon)_j$. 

We will now argue by contradiction. So let us assume that $y(\varepsilon)_k\geq x(\varepsilon)_k$ for all $k\in \Z^d$. Then we find by interpolation that 
\begin{align}\nonumber
0  & = V'(y(\varepsilon)_i) - V'(x(\varepsilon)_i)  + \varepsilon R_i(y(\varepsilon)) - \varepsilon  R_i(x(\varepsilon)) \\ \nonumber  & = \left( \int_0^1V''(t y(\varepsilon)_i+(1-t) x(\varepsilon)_i)dt \right)  (y(\varepsilon)_i  -  x(\varepsilon)_i)  \\ \nonumber & + \varepsilon \!\!\!\!\!\! \sum_{{\tiny \begin{array}{c} ||j-i||\leq r\\ ||k-j||\leq r \end{array}} }\!\!\!\!\!\!  \left( \int_0^1 \p_{i,k}S_j(t y(\varepsilon)_i+(1-t) x(\varepsilon)_i) dt\right)  ( y(\varepsilon)_k - x(\varepsilon)_k)  \nonumber \\ 
 & \leq \left( -c + \varepsilon  C_2/(2r+1)^d \right)  (y(\varepsilon)_i -  x(\varepsilon)_i ) \, .
\nonumber
\end{align}
Here, in the inequality we have used that $V''\leq -c$, that $\p_{i,k}S_j\leq 0$ for all $i\neq k$, that $y(\varepsilon)_k-x(\varepsilon)_k\geq 0$ and that $|\p_{i,i}S_j|\leq C_2/(2r+1)^{2d}$.

This estimate implies that if we choose $0<\varepsilon_1'< \min\{c(2r+1)^d/C_2, \varepsilon_0\}$ and $0\leq \varepsilon\leq \varepsilon_1$, then $x(\varepsilon)_i \geq y(\varepsilon)_i$. But we assumed that $x(\varepsilon)_k\leq y(\varepsilon)_k$ for all $k\in \Z^d$, and therefore it must hold that $x(\varepsilon)_i = y(\varepsilon)_i$. By the comparison principle of Lemma \ref{maximumprinciple} that holds for $\varepsilon>0$, this implies that $x(\varepsilon)=y(\varepsilon)$. But then also $x=y$ by uniqueness, which contradicts our assumptions.
\end{proof}
Note that the $\varepsilon_1$ in the statement of Theorem \ref{conversethm} may not be the same as the $\varepsilon_1$ in the statement of Theorem \ref{corollaryordering}.
%Finally, we prove that the conclusions of Theorems \ref{ordered recurrence} and \ref{conversethm} are summarized in Corollary \ref{corollaryordering} at the begining of this section.
%\\ \mbox{} \\
%{\it Proof} (of Corollary \ref{corollaryordering}):  
%If $x$ is Birkhoff such that every $x_i$ is a local minimum of $V$, then it follows by Theorem \ref{ordered recurrence} that $\Gamma(x(\varepsilon))$ is a lamination of solutions to (\ref{FKepsilon}). If $x$ is Birkhoff, but $x_i$ are not all local minima, then Theorem \ref{conversethm} shows that $\Gamma_\varepsilon(x)$ is not a lamination. Finally, if $x$ is not Birkhoff, then neither is $x(\varepsilon)$. 
%\hfill $\square$

We finish this section with the result that there are no Birkhoff solutions to (\ref{FKepsilon2}) close to the anti-continuum limit other than the ones already found in Theorem \ref{continuation}.  Again, the $\varepsilon_1$ of Theorem \ref{birkhoffcontinuation} may be different from the $\varepsilon_1$'s in Theorems \ref{corollaryordering} and \ref{conversethm}. %Recall that any Birkhoff configuration has bounded oscillation.

\begin{theorem}\label{birkhoffcontinuation}
Let $\omega\in\R^d$, let $K>0$ be so that ${\rm osc}\, (X)\leq K$ for all Birkhoff configurations $X$ of rotation vector $\omega$ and let $\delta_0, \varepsilon_0>0$ be as in Theorem \ref{continuation}. Then there is an $0<\varepsilon_1\leq \varepsilon_0$ so that for all $0\leq \varepsilon\leq \varepsilon_1$ the following is true: 

When $x^{\varepsilon}$ is any Birkhoff solution to (\ref{FKepsilon2}) of rotation vector $\omega$, then $x^{\varepsilon}\in B_{\delta_0}(x)$ for some Birkhoff solution $x$ of (\ref{antilimit}) and hence $x^{\varepsilon}=x(\varepsilon)$. %Moreover, if $x(\varepsilon)$ is Birkhoff, then $x$ was already Birkhoff. %If $x^{\varepsilon}$ is Birkhoff, then so is $x$.  
\end{theorem}
\begin{proof}
Because $V$ is Morse, we have that for all $\delta>0$ there is an $\varepsilon>0$ so that if $|V'(X_i)|<\varepsilon$, then $|X_i - x_i| < \delta$ for some critical point $x_i$ of $V$.

Now suppose that $x^{\varepsilon}$ is Birkhoff with rotation vector $\omega$ and solves (\ref{FKepsilon2}). Then ${\rm osc}(x^{\varepsilon})\leq K$ and hence, as in the proof of Theorem \ref{continuation}, there is a constant $C_1$ depending only on $K$ so that
 $$|V'(x^{\varepsilon}_i)|=|\varepsilon R_i(x^{\varepsilon})|\leq \varepsilon \, C_1\ \mbox{uniformly for}\ i\in \Z^d\, .$$ 
Thus, there is an $0<\varepsilon_1\leq \varepsilon_0$ so that if $0\leq \varepsilon\leq \varepsilon_1$, then there are critical points $x_i$ of $V$ such that $|x^{\varepsilon}_i-x_i|< \delta_0$ for all $i$. In other words, for such $\varepsilon$ we have that $x^{\varepsilon}\in B_{\delta_0}(x)$. Because $x(\varepsilon)$ is the unique solution to (\ref{FKepsilon2}) in $B_{\delta_0}(x)$, this implies that $x^{\varepsilon}=x(\varepsilon)$.
 
It is clear that the solution $x$ to (\ref{antilimit}) for which $x^{\varepsilon}=x(\varepsilon)$ must be Birkhoff, because if it is not, $x(\varepsilon)\in B_{\delta_0}(x)$ cannot be Birkhoff either.
\end{proof}

\section{Local action minimizers}\label{minimizersection}
In this section we shall prove Theorem \ref{corollaryordering}. To do this, we recall from Section \ref{setting} that the solutions to (\ref{FKepsilon2}) are precisely the stationary points of the formal action
$$  W^{\varepsilon}(x):=\sum_{j\in \Z^d} S_j^{\varepsilon}(x)\ \mbox{with}\ S_j^{\varepsilon}(x) = V(x_j)+\varepsilon S_j(x)\, .$$
%$$W^{\varepsilon}(x):=  \sum_{j\in \Z^d} S_j^{\varepsilon}(x)\ \mbox{with}\ S_j^{\varepsilon}(x):=V(x_j) + \varepsilon S_j(x)\ \mbox{and the}\ S_j\ \mbox{satisfying conditions {\bf A}-{\bf C}}\, .$$
%Indeed, the derivative of $W^{\varepsilon}(x)$ with respect to $x_i$ is precisely 
%$$\frac{\p W^{\varepsilon}(x)}{\p x_i} = \sum_{j\in \Z^d}\p_iS_j^{\varepsilon}(x)= V'(x_i)+\varepsilon \sum_{j\in \Z^d} \p_iS_j(x)= V'(x_i)+\varepsilon R_i(x)\, .$$ 
%This explains why the solutions to (\ref{FKepsilon2}) are also called ``stationary configurations''. 
A problem is of course that this formal action in general does not define a convergent sum.
%Indeed, a configuration $x:\Z^d\to\R$ is a solution to (\ref{FKepsilon2}) if and only if for all $v:\Z^d\to\R$ with finite support it holds that
%\begin{align}\label{vareqn}
%\left.\frac{d}{dh}\right|_{h=0} \!\!\! W^{\varepsilon}(x+hv) = \lim_{h\to 0}\frac{1}{h} \sum_{j\in \Z^d} \left(S^{\varepsilon}_{j}(x+hv)-S^{\varepsilon}_j(x)\right) =0\, .
%\end{align}
% or equivalently, when $\widetilde \Gamma(x(\varepsilon))$ is ordered.% If it is, then it follows rather easily that
%$%$\widetilde{\Gamma}_{\varepsilon}(x)=\widetilde{\Gamma}(x(\varepsilon))\ \mbox{and} \ \Gamma_{\varepsilon}(x) = \Gamma(x(\varepsilon))\ .$$
%Here, it is important to note that the derivative in (\ref{vareqn}) is well-defined: because $v$ has finite support and $S_j$ has finite range $r$, the sum in (\ref{vareqn}) is finite. In view of (\ref{vareqn}) the solutions to (\ref{FKepsilon}) are also called stationary configurations. %A special kind of stationary configurations are those that ``minimize'' the formal action:
%\begin{definition}
%We say that $x$ is a {\it global minimizer} of (\ref{FKepsilon}) if for all $v:\Z^d\to\R$ with finite support it holds that 
%$$W^{\varepsilon}(x+v)-W^{\varepsilon}(x) = \sum_{i\in \Z^d} \left( S^{\varepsilon}_i(x+v) - S_i^{\varepsilon}(x)\right) \geq 0\, .$$
%\end{definition}
%Global minimizers are clearly stationary. 
Thus, in order to make this formal variational principle useful, let us make a few definitions. First of all, we define for any finite subset $B\subset \Z^d$ the finite action
$$W_B^{\varepsilon}(x):=\sum_{j\in B} S_j^{\varepsilon}(x)\, .$$ 
It is clear that $W_B^\varepsilon(x)$ is a finite sum and is a function of only those $x_i$ for which 
$\min_{j\in B}||j-i||\leq r$. On the other hand, when $i$ is such that $\min_{j\in \Z^d \backslash B}||j-i||> r$, then 
$$\frac{\p W_B^{\varepsilon}(x)}{\p x_i} = \sum_{j\in B} \p_{i}S_{j}^{\varepsilon}(x) = \sum_{j\in\Z^d} \p_iS_j^{\varepsilon}(x) = V'(x_i)+\varepsilon R_i(x) \, .$$
In view of this, we define the exterior and the interior of $B$ as
$$\overline{B}:=\{i \in \Z^d \, | \, \min_{j\in B}||j-i||\leq r\} \ \mbox{and}\ \mathring{B}:=\{i\in \Z^d \, | \, \min_{j\in \Z^d\backslash B}||j-i||> r\}\, .% \ \mbox{and}\ \p B:=\overline{B}\backslash \mathring{B}\, .
$$
These definitions are such that $x$ solves (\ref{FKepsilon2}) if and only if for every finite $B\subset \Z^d$ and every $v:\Z^d\to\R$ with ${\rm supp}(v)\subset \mathring{B}$ it holds that 
$$\left. \frac{d}{dt}\right|_{t=0} \!\!\! W_B^{\varepsilon}(x+tv)=0\, .$$
\noindent The first result of this section is of a technical nature:
\begin{lemma}\label{confinedcontinuation}
Let $K>0$ and let $\delta_0, \varepsilon_0>0$ be as in Theorem \ref{continuation}. Then there is an $0<\varepsilon_1\leq \varepsilon_0$ so that for all $0\leq \varepsilon\leq \varepsilon_1$ the following is true: 

Let $x$ be a solution of (\ref{antilimit}) of bounded oscillation ${\rm osc}(x)\leq K$ and assume that all $x_i$ are local minima of $V$, let $z\in B_{\delta_0}(x)$, let $B\subset \Z^d$ be finite and define
$$A_{\delta_0,B,x,z} :=\{v:\Z^d\to \R \, |\, {\rm supp}(v)\subset \mathring{B}\ \mbox{and}\ z+v\in B_{\delta_0}(x)\, \}\, .$$
Then there exists a unique configuration $v_{B}^{\varepsilon}(z)\in A_{\delta_0,B,x,z}$ such that  
$$\left. \frac{d}{dt}\right|_{t=0} \!\!\! W_B^{\varepsilon}(z+v^{\varepsilon}_B(z)+tw)= \sum_{i\in \mathring{B}} \left( V'(z_i+v^{\varepsilon}_B(z)_i)+\varepsilon R_i(z+v^{\varepsilon}_B(z)) \right) w_i= 0$$
 for all $w:\Z^d\to\R$ with ${\rm supp}(w)\subset \mathring{B}$.
 
Furthermore, the map 
\begin{align}
\nonumber
& d^{\varepsilon}_{B,z} : v\mapsto W^{\varepsilon}_{B}(z+v) - W_{B}^{\varepsilon}(z) \ \mbox{defined on}\ A_{\delta_0,B,x,z}  \nonumber
\end{align}
is convex and $v_{B}^{\varepsilon}(z)$ is its unique minimizer. The map
$$D_{B}^{\varepsilon}: z\mapsto  W^\varepsilon_{B}(z+v_{B}^{\varepsilon}(z)) - W^\varepsilon_{B}(z) %= \min_{v\in A_{\delta_0,B,x,z}} d^{\varepsilon}_{B,z} (v)
\leq 0\, $$
is continuous for pointwise convergence and $D^\varepsilon_{B}(z)= 0$ if and only if $z$ solves (\ref{FKepsilon2}) on $\mathring{B}$.

%Moreover, the function $$D^\varepsilon_B: z\mapsto W^\varepsilon_B(z+v_B^{\varepsilon}(z)) - W^\varepsilon_B(z) \ \mbox{defined on} \ B_{\delta_0}(x)$$ %= \min_{v\in A_{\delta_0, B, x, z}} W^{\varepsilon}_B(z+v) - W^{\varepsilon}_B(z)\leq 0 $$ 
 % $v^B(y)\equiv0$.

% it holds for any configuration $w$ with $\text{supp}(w)\subset \mathring{B}$ and such that $z+w\in B_{\delta_0}(x)$ that $$W^\varepsilon(z+w)-W^\varepsilon(z+v^B(z))\geq 0,$$ i.e. $z+v^B(z)$ is a local minimum of $W^\varepsilon_B$ with fixed boundary conditions. 
%Also, when all $x_i$ are local minima of $V$ and $B=B_{r_3}(0)$ is a ball, then
\end{lemma}

\begin{proof}
The proof of the first statement is similar to that of Theorem \ref{continuation}. In fact, we can simply adapt (\ref{defF}) and define 
\begin{equation*} F_{B,x,z}(X,\varepsilon)_i:=\left\{ \begin{array}{ll} \p W^{\varepsilon}_B(X)/\p x_i = V'(X_i)+\varepsilon R_i(X) & \text{if}\ i\in \mathring{B}, \\ z_i-X_i & \text{if} \ i\notin \mathring{B}.\end{array}\right.\end{equation*}
Then $F_{B,x,z}(X,\varepsilon)=0$, if and only if $X_i=z_i$ for all $i\in \Z^d \backslash \mathring{B}$ and $V'(X_i)+\varepsilon R_i(x)=0$ for all $i\in \mathring{B}$. If we denote by 
$x_{B,x,z}$ the configuration defined by 
$$(x_{B,x,z})_i=\left\{ \begin{array}{ll} x_i & \text{if}\ i\in \mathring{B}\, , \\ z_i& \text{if} \ i\notin \mathring{B}\, ,\end{array}\right.$$ 
then it is clear that $F_{B,x,z}(x_{B,x,z},0)=0$ and that the Fr\'echet derivative $D_X  F_{B,x,z}(x_{B,x,z}, 0)$ has a bounded inverse. In fact, the resulting contraction operator is given by $$K_{\varepsilon, B,x,z}(X)_i:= \left\{ \begin{array}{ll} X_i - \frac{V'(X_i) + \varepsilon R_i(X)}{V''(x_i)} & \text{if}\ i\in \mathring{B}\, , \\ z_i& \text{if} \ i\notin \mathring{B}\, .\end{array}\right.$$ 
This makes it clear that the contraction constant of $K_{\varepsilon, B,x,z}$ is at least as small as that of the contraction $K_{\varepsilon, x}$ defined in (\ref{defF}) and that $K_{\varepsilon, B,x,z}$ maps $B_{\delta_0}(x)$ into itself if $K_{\varepsilon, x}$ does. Thus, for $0\leq \varepsilon \leq \varepsilon_0$ our map $K_{\varepsilon, B,x,z}$ is a contraction on $B_{\delta_0}(x)$. What results is a unique fixed point $z+v_{B}^{\varepsilon}(z)\in B_{\delta_0}(x)$ so that $v_{B}^{\varepsilon}(z)\in A_{\delta_0,B,x,z}$. Clearly, $v_{B}^{\varepsilon}(z)$ depends continuously on $z$. %As a consequence, also the map $D_{\delta_0, B,x}^{\varepsilon}$ is continuous.

Since $v^{\varepsilon}_B(z)$ is unique, it is actually the unique stationary point of the map $$v\mapsto d_{B,z}^{\varepsilon}(v) :=W_B^{\varepsilon}(z+v)- W_B^{\varepsilon}(z)\ \mbox{defined on} \ A_{\delta_0,B,x,z}\, .$$
Now we use that all $x_i$ are local minima of $V$: because $V$ is a Morse function, we may assume that $V''\geq c$ for some $c>0$ on $(x_i-\delta_0, x_i+\delta_0)$ for all $i\in \Z^d$. In turn, because $|\p_{i,k}S_j(X)|\leq C_2/(2r+1)^{2d}$ is uniformly bounded and $w_iw_k\leq \frac{1}{2}(w_i^2+w_k^2)$, this implies that 
\begin{align}\nonumber
 \left. \frac{d^2}{dt^2}\right|_{t=0}\!\!\!\!\!\!\! d^{\varepsilon}_{B,z}(v+tw) = 
\sum_{j\in \mathring{B}} V''(z_j+v_j)w_j^2  + \varepsilon \!\!\!\!\! \sum_{\tiny \begin{array}{c} j\in B \\ i,k\in \mathring{B}\end{array} } \!\!\!\!\! \p_{i,k}S_j(z+v) w_iw_k  \\
\geq c \sum_{j\in \mathring{B}} w_j^2  - \frac{\varepsilon C_2}{2(2r+1)^{2d}} \!\!\!\!\!\!\!\!\! \sum_{\tiny \begin{array}{c} j\in B \\ i,k\in \mathring{B}\cap B_{r}(j)\end{array} }\!\!\!\!\!\!\!\! (w_i^2+w_k^2) 
 \! \geq  \left(c - \frac{\varepsilon C_2}{2}\right) \sum_{j\in \mathring{B}} w_j^2\, . \nonumber
%& t W_B^{\varepsilon}(z+v_1) + (1-t) W_B^{\varepsilon}(z+v_2) - W_{B}^{\varepsilon}(z+tv_1+(1-t)v_2)  = \\ \nonumber
%& \sum_{j\in B}\p_{i,k} S_j^{\varepsilon}(X) \geq \sum \p_{i,i}S_j^{\varepsilon}\, .
\end{align} 
We conclude that $d_{B,z}^{\varepsilon}$
is convex as soon as $0\leq \varepsilon \leq \varepsilon_1 :=\min\{ 2c/C_2, \varepsilon_0\}$. Because $v=v_{B}^{\varepsilon}(z)$ is the unique stationary point of this map, it is also its unique minimizer. 

The remaining statements of the lemma follow immediately.
%It is clear that $D^\varepsilon_B(z)$ is then well defined and by following the proof of Theorem \ref{continuous}, it is not difficult to see that it is continuous. Moreover, since $v^B(z)$ is the unique local minimizing solution, it holds that $D^\varepsilon_B(z)=0$ if and only if $v^B(z)\equiv 0$, so $z$ solves the equation (\ref{FKepsilon}) for all $i\in \mathring{B}$.
\end{proof}
\noindent The quantity $D^{\varepsilon}_B(z)$ measures the amount by which $z$ fails to be a local action minimizer among configurations that are equal to $z$ outside $\mathring{B}$ and are $\delta_0$-close to $x$. One of its simple and important properties is that it is subadditive as a function of the set $B$:

\begin{lemma}\label{subadditive}
Assume the conditions of Lemma \ref{confinedcontinuation} are satisfied. When $B_1, \ldots, B_n\subset B$ are pairwise disjoint finite sets, then
$$D^{\varepsilon}_B(z)\leq D^{\varepsilon}_{B_1}(z)+ \ldots + D^{\varepsilon}_{B_n}(z)\, .$$
\end{lemma}
\begin{proof}
Assume first that $B_1 \cup B_2 = B$ and $B_1\cap B_2=\emptyset$. Then $W_{B}^{\varepsilon} = W_{B_1}^{\varepsilon} + W_{B_2}^{\varepsilon}$ and hence
\begin{align}\nonumber
 D^{\varepsilon}_B(z) & =  \min_{v\in A_{\delta_0, B, x,z}} \left( W_{B_1}^{\varepsilon}(z+v)- W_{B_1}^{\varepsilon}(z) +W_{B_2}^{\varepsilon}(z+v)- W_{B_2}^{\varepsilon}(z)\right)   \\
 %\nonumber & \leq  \min_{v\in A_{\delta_0, B_1, x,z}\cup A_{\delta_0, B_2, x,z}} \left( W_{B_1}^{\varepsilon}(z+v)- W_{B_1}^{\varepsilon}(z) +W_{B_2}^{\varepsilon}(z+v)- W_{B_2}^{\varepsilon}(z)\right)  
  \nonumber & \leq \min_{v\in A_{\delta_0, B_1, x,z}} \left( W_{B_1}^{\varepsilon}(z+v)- W_{B_1}^{\varepsilon}(z) \right) + \min_{v\in A_{\delta_0, B_2, x,z}} \left( W_{B_2}^{\varepsilon}(z+v)- W_{B_2}^{\varepsilon}(z)\right)  \\ \nonumber  & = D^{\varepsilon}_{B_1}(z) + D^{\varepsilon}_{B_2}(z)\, .
\end{align}
The lemma now follow by induction, and from the fact that $D^{\varepsilon}_{B\backslash (B_1\cup \ldots \cup B_n)}\leq 0$.
\end{proof}

\noindent The following lemma describes what happens if the continuation of a Birkhoff configuration is not Birkhoff. To formulate it, let us define for any two configurations $x, y:\Z^d\to\R$ the configurations $x\wedge y$ and $x\vee y$ as
$$(x\wedge y)_i:=\min\{x_i, y_i\}\ \mbox{and} \ (x \vee y)_i:=\max\{x_i, y_i\}\, .$$

\begin{lemma}\label{uniform_aubry}
Let $x$ be a recurrent Birkhoff solution of (\ref{antilimit}) of rotation vector $\omega \in \R^d\backslash \Q^d$ so that all $x_i$ are local minima of $V$. Moreover, let $\varepsilon_1>0$ be as in Lemma \ref{confinedcontinuation} and $0< \varepsilon\leq \varepsilon_1$. Assume that the continuations $x(\varepsilon)$ and $y(\varepsilon):=\tau_{k,l}x(\varepsilon)$ are not ordered. Then there exist an integer $r_1$ and a positive constant $\gamma>0$, such that on every ball $B_{r_1}(j)$ it holds that 
$$\left| D^\varepsilon_{B_{r_1}(j)}(x(\varepsilon) \wedge y(\varepsilon))\right| \geq \gamma\ \text{and} \ \left| D^\varepsilon_{B_{r_1}(j)}(x(\varepsilon) \vee y(\varepsilon))\right| \geq \gamma\, .$$
\end{lemma}

\begin{proof} %As in the proof of Lemma \ref{finite balls}, we write $y=\tau_{k,l}x$ so that $y(\varepsilon)=(\tau_{k,l}x)(\varepsilon)=\tau_{k,l}(x(\varepsilon))$. %Let $r_0$ be as in the conclusion of Lemma \ref{finite balls} and define $r_1:= r_0 + r$, where $r$ is the range of interaction of the local potentials. 
Let us write $y=\tau_{k,l}x$ so that $y(\varepsilon)=(\tau_{k,l}x)(\varepsilon)=\tau_{k,l}(x(\varepsilon))$ (by uniqueness) and let us assume that $x=x^-(0)$ (recall the definition in Theorem \ref{xplusminuslemma}; the proof is similar when $x=x^+(0)$). Then $y=x^-(\omega\cdot k + l)$.

We shall pursue a proof by contradiction, so assume that the first conclusion of the lemma does not hold. Then there exists, for every $n>0$, a sequence of indices $m\mapsto k_m^n$ such that $$\lim_{m\to \infty} \left| D^\varepsilon_{B_{n}(k_m^n)}(x(\varepsilon)\wedge y(\varepsilon))\right| = 0 \, .$$ 
For an appropriate diagonal sequence $k_n:=k^n_{m(n)}$ it then holds that 
$$\lim_{n\to \infty} \left| D^\varepsilon_{B_{n}(k_n)}(x(\varepsilon)\wedge y(\varepsilon))\right| = 0 \, .$$ 
It follows from Lemma \ref{subadditive} that $\left| D^\varepsilon_{B_{m}(k_n)}(x(\varepsilon)\wedge y(\varepsilon))\right| \leq \left| D^\varepsilon_{B_{n}(k_n)}(x(\varepsilon)\wedge y(\varepsilon))\right|$ for every $m\leq n$ and hence, for every $m>0$,
$$\lim_{n\to \infty} \left| D^\varepsilon_{B_{m}(k_n)}(x(\varepsilon)\wedge y(\varepsilon))\right| = 0 \, .$$ 
Equivalently, for $l_n\in \Z$ arbitrary, and because the translate of the minimum is the minimum of the translates,
$$\lim_{n\to \infty} \left| D^\varepsilon_{B_{m}(0)}(\tau_{k_n,l_n}x(\varepsilon)\wedge \tau_{k_n,l_n}y(\varepsilon))\right| = 0 \, .$$ 
Choosing $l_n$ in such a way that $(\tau_{k_n,l_n}x(\varepsilon))_0 = x(\varepsilon)_{k_n}+l_n \in [0,1]$, and hence $y(\varepsilon)_{k_n}+l_n$ uniformly bounded, we may assume (by going twice to a subsequence and using the compactness of the set of Birkhoff configurations of a given rotation vector that follows from Tychonov's theorem) that $\tau_{k_n,l_n}x(\varepsilon)\to x_{\infty}(\varepsilon)$ and that $\tau_{k_n,l_n}y(\varepsilon) \to y_{\infty}(\varepsilon)$.
%As in the proof of Lemma \ref{finite balls}, we can pass to a subsequence so that $\tau_{k_{n},l_{n}}x(\varepsilon) \to x_{\infty}(\varepsilon)$ and $\tau_{k_{n},l_{n}}y(\varepsilon) \to y_{\infty}(\varepsilon)$. 
Then it holds in particular that 
$$\tau_{k_n, l_n}( x(\varepsilon)\wedge y(\varepsilon)) = \tau_{k_{n}, l_{n}}x(\varepsilon)\wedge \tau_{k_{n},l_{n}}y(\varepsilon) \to x_{\infty}(\varepsilon)\wedge y_{\infty}(\varepsilon)\, $$ and by continuity of $D^\varepsilon_{B_{m}(0)}$ it follows that $D_{B_{m}(0)}(x_{\infty}(\varepsilon)\wedge y_{\infty}(\varepsilon))=0$.

By Lemma \ref{confinedcontinuation} this means that $x_{\infty}(\varepsilon)\wedge y_{\infty}(\varepsilon)$ solves the equations (\ref{FKepsilon2}) on $\mathring{B}_{m}(0) = B_{m-r}(0)$. This holds for every $m$, so $x_{\infty}(\varepsilon)\wedge y_{\infty}(\varepsilon)$ solves (\ref{FKepsilon2}) everywhere. But so does $x_{\infty}(\varepsilon)$: it is the pointwise limit of solutions and hence a solution itself. Because $x_{\infty}(\varepsilon)\wedge y_{\infty}(\varepsilon) \leq x_{\infty}(\varepsilon)$, we thus conclude from the comparison principle (see Lemma \ref{maximumprinciple} in Appendix \ref{proofappendix}) that either $x_{\infty}(\varepsilon)\wedge y_{\infty}(\varepsilon) = x_{\infty}(\varepsilon)$ or $x_{\infty}(\varepsilon)\wedge y_{\infty}(\varepsilon) \ll x_{\infty}(\varepsilon)$. In both cases it follows that $x_{\infty}(\varepsilon)$ and $y_{\infty}(\varepsilon)$ are ordered (in the first case $x_{\infty}(\varepsilon) \leq y_{\infty}(\varepsilon)$ and hence by the comparison principle either $x_{\infty}(\varepsilon)\ll y_{\infty}(\varepsilon)$ or $x_{\infty}(\varepsilon)= y_{\infty}(\varepsilon)$ and in the second case $y_{\infty}(\varepsilon)\ll x_{\infty}(\varepsilon)$). 

Recall from Theorem \ref{continuous} that $\Phi_{\varepsilon}: x\mapsto x(\varepsilon)$ is a homeomorphism from $\Gamma(x)$ to $\Gamma_{\varepsilon}(x)$. As a result, because $\tau_{k_n, l_n}x(\varepsilon)\to x_{\infty}(\varepsilon)$, also $\tau_{k_n,l_n}x\to x_{\infty} \in \Gamma(x)$ and similarly $\tau_{k_n,l_n}y\to y_{\infty}\in \Gamma(x)$.

 Moreover, let $s:=\lim_{n\to \infty} \omega \cdot k_n +l_n$ so that $x_{\infty}=x^{\pm}(s)$ and $y_{\infty}=x^{\pm}(s+ \omega\cdot k + l)$. If we then choose any sequence $(K_n, L_n)$ with $\omega\cdot K_n + L_n \nearrow -s$ then it follows that $\tau_{K_n,L_n} x_{\infty}\to x^-(0) = x$ and $\tau_{K_n,L_n} y_{\infty}\to x^-(\omega\cdot k + l)=y$. Using the homeomorphism $\Phi_{\varepsilon}$ again, we find that $\tau_{K_n, L_n}x_{\infty}(\varepsilon)\to x(\varepsilon)$ and $\tau_{K_n,L_n}y_{\infty}(\varepsilon)\to y(\varepsilon)$. Because $x_{\infty}(\varepsilon)$ and $y_{\infty}(\varepsilon)$ are ordered, this implies that also $x(\varepsilon)$ and $y(\varepsilon)$ are ordered, contrary to our assumption. 

%As in the proof of Lemma \ref{finite balls}, this implies in turn that $x(\varepsilon)$ and $y(\varepsilon)$ are ordered, contrary to our assumptions. 

A similar reasoning proves the statement for $x(\varepsilon)\vee y(\varepsilon)$.
%More precisely, we will show that if the conclusion of the lemma does not hold, then $\Gamma_\varepsilon(x)$ is ordered. 
 %So let us assume that the conclusion of the lemma does not hold. Then there exist a sequence of indices $k_n$ and balls $B_n(k_n)$ of radius $n$ such that for every $j \in B_{n}(k_n)$, $x(\varepsilon)_j\leq y(\varepsilon)_j$ (again, the proof is similar when $x(\varepsilon)_j\geq y(\varepsilon)_j$ on these balls). 
% Equivalently, it holds that $\tau_{k_n,0}x(\varepsilon)\leq \tau_{k_n,0}y(\varepsilon)$ on $B_n(0)$.  By construction, we know that $x_{\infty}(\varepsilon) \leq y_{\infty}(\varepsilon)$. 
\end{proof}

\noindent We are now ready to prove Theorem \ref{corollaryordering}.

\begin{proofof}{Theorem \ref{corollaryordering}} 
Let $\varepsilon_1>0$ be as in Lemma \ref{uniform_aubry} and $0<\varepsilon\leq \varepsilon_1$. Let us write $y=\tau_{k,l}x$ so that $y(\varepsilon)=\tau_{k,l}x(\varepsilon)$ and assume without loss of generality that $x\leq y$. Recall that $x(\varepsilon)\in B_{\delta_0}(x)$ and $y(\varepsilon)\in B_{\delta_0}(y)$. It is important to remark that then also the minimum $x(\varepsilon)\wedge y(\varepsilon)\in B_{\delta_0}(x)$ and the maximum $x(\varepsilon)\vee y(\varepsilon)\in B_{\delta_0}(y)$.

Suppose now that $x(\varepsilon)$ and $y(\varepsilon)$ are not ordered. It then follows from Lemma \ref{uniform_aubry} that $|D^{\varepsilon}_{B_{r_1}(j)}(x(\varepsilon)\wedge y(\varepsilon))|>\gamma$ and $|D^{\varepsilon}_{B_{r_1}(j)}(x(\varepsilon)\vee y(\varepsilon))|>\gamma$ for every $j\in \Z^d$, some integer $r_1$ and some $\gamma>0$. It is easy to prove that there exist an integer $r_2>0$ and an $\alpha>0$ so that any ball of radius $r_3>r_2$ contains at least $\alpha r_3^d$ disjoint balls of radius $r_1$. In view of Lemma \ref{subadditive}, we conclude that for any such $r_3\geq r_2$ it holds that 
$$\left| D_{B_{r_3}(0)}(x(\varepsilon)\wedge y(\varepsilon)) \right| > \alpha \gamma r_3^d\ \mbox{and}\ \left| D_{B_{r_3}(0)}(x(\varepsilon)\vee y(\varepsilon)) \right| > \alpha \gamma r_3^d\, .$$ 
This means that on large enough balls, $x(\varepsilon)\wedge y(\varepsilon)\in B_{\delta_0}(x) $ and $x(\varepsilon)\vee y(\varepsilon)\in B_{\delta_0}(y)$ are extremely far from being local minimizers. In particular, if we write $$m(\varepsilon) := x(\varepsilon)\wedge y(\varepsilon)+v^{\varepsilon}_{B_{r_3}(0)}(x(\varepsilon)\wedge y(\varepsilon))\ \mbox{and}\ M(\varepsilon) := x(\varepsilon)\wedge y(\varepsilon)+v^{\varepsilon}_{B_{r_3}(0)}(x(\varepsilon)\wedge y(\varepsilon))\, ,$$
then it holds that
\begin{align}\label{estim1}
W^\varepsilon_{B_{r_3}(0)}\left(m(\varepsilon) \right) \!+\! W^\varepsilon_{B_{r_3}(0)}\left(M(\varepsilon) \right)\! \leq\!
 W^\varepsilon_{B_{r_3}(0)}(x(\varepsilon) \wedge y(\varepsilon)) \!+ \! W^\varepsilon_{B_{r_3}(0)}(x(\varepsilon) \vee y(\varepsilon))  \!-\! 2 \alpha \gamma r_3^d\, .
\end{align}
On the other hand, due to the minimum-maximum property (see Lemma \ref{Aubry} in Appendix \ref{proofappendix}), 
\begin{align}\label{estim2}
 W^\varepsilon_{B_{r_3}(0)}(x(\varepsilon) \wedge y(\varepsilon)) + W^\varepsilon_{B_{r_3}(0)}(x(\varepsilon) \vee y(\varepsilon)) \leq W^\varepsilon_{B_{r_3}(0)}(x(\varepsilon)) + W^\varepsilon_{B_{r_3}(0)}(y(\varepsilon))\, .
\end{align}
This in turn suggests that $x(\varepsilon)$ or $y(\varepsilon)$ should not be a local minimizer either. To make this intuition rigorous, let us define the configurations
$\tilde x^{\varepsilon}$ and $\tilde y^{\varepsilon}$ by 
$$\tilde x^{\varepsilon}_i:=\left\{ \begin{array}{ll} m(\varepsilon)_i & \mbox{for}\ i \in \mathring{B}_{r_3}(0)  \\ x(\varepsilon)_i & \mbox{for}\ i \notin \mathring{B}_{r_3}(0) \end{array} \right. \mbox{and}\ \ \tilde y^{\varepsilon}_i:=\left\{ \begin{array}{ll} M(\varepsilon)_i & \mbox{for}\ i \in \mathring{B}_{r_3}(0)  \\ x(\varepsilon)_i & \mbox{for}\ i \notin \mathring{B}_{r_3}(0)\end{array} \right. \, .$$ 
Then $\tilde x^{\varepsilon} \in B_{\delta_0}(x)$ and $\tilde y^{\varepsilon} \in B_{\delta_0}(y)$. Moreover,
$\tilde x^{\varepsilon}$ is a variation of $x(\varepsilon)$ supported in $\mathring{B}_{r_3}(0)$ and $\tilde y^{\varepsilon}$ is a variation of $y(\varepsilon)$ supported in $\mathring{B}_{r_3}(0)$. In addition, it is not hard to prove that 
\begin{align}\label{estim3}
\left| W^{\varepsilon}_{B_{r_3}(0)}(m(\varepsilon))-W^\varepsilon_{B_{r_3}(0)}(\tilde x^{\varepsilon}) \right| \leq \beta r_3^{d-1}\ \mbox{and} \ \left| W^{\varepsilon}_{B_{r_3}(0)}(M(\varepsilon))-W^\varepsilon_{B_{r_3}(0)}(\tilde x^{\varepsilon})\right| \leq \beta r_3^{d-1} 
\end{align}
for some constant $\beta>0$, because $m(\varepsilon)_i=\tilde x^{\varepsilon}_i$ and $M(\varepsilon)_i=\tilde y^{\varepsilon}_i$ for all $i\in \overline{B}_{r_3}(0)$ except those outside $\mathring{B}_{r_3}(0)$. Together, estimates (\ref{estim1}), (\ref{estim2}) and (\ref{estim3}) yield that 
$$W^\varepsilon_{B_{r_3}(0)}(\tilde x^{\varepsilon})+W^\varepsilon_{B_{r_3}(0)}(\tilde y^{\varepsilon}) \leq W^\varepsilon_{B_{r_3}(0)}(x(\varepsilon)) + W^\varepsilon_{B_{r_3}(0)}(y(\varepsilon))  -2\alpha \gamma r_3^d+2 \beta r_3^{d-1}\, .$$ 
Hence, if $r_3$ is large enough, it must hold that 
$$\mbox{either}\ W^\varepsilon_{B_{r_3}(0)}(\tilde x^{\varepsilon})< W^\varepsilon_{B_r(0)}(x(\varepsilon))\ \mbox{or}\ W^\varepsilon_{B_{r_3}(0)}(\tilde y^{\varepsilon})<W^\varepsilon_{B_{r_3}(0)}(y(\varepsilon))\, .$$ This contradict the fact that $x(\varepsilon)$ and $y(\varepsilon)$ must both be local minimizers with respect to variations in $A_{\delta_0, B_{r_3}(0), x, x(\varepsilon)}$ and $A_{\delta_0, B_{r_3}(0), y, y(\varepsilon)}$ respectively, according to Lemma \ref{confinedcontinuation}.
 \end{proofof}

 \section{Proof of Theorem \ref{mainthm}}\label{proofsection}

In this section we combine the results obtained so far to finish the proof of Theorem \ref{mainthm}. Throughout this section, $\omega\in\R^d\backslash \Q^d$ will be fixed. We shall denote by $N$ the number of geometrically distinct local minima of $V$ and by $\Delta_{N-1}$ the $(N-1)$-dimensional simplex. %Theorem \ref{mainthm} says that for any irrational rotation vector $\omega\in \R^d\backslash \Q^d$, the set of laminations of solutions to (\ref{FKepsilon2}) is homeomorphic to $\Delta_{N-1}$.

Our strategy will be to associate to any $p\in \Delta_{N-1}$ a lamination of solutions to (\ref{FKepsilon2}) of rotation vector $\omega$. By Theorem \ref{clas2} and Theorem \ref{cross2} any such lamination is characterized by a unique Radon probability measure on $\R/\Z$. Hence our map will be of the form 
$$\Psi_{\varepsilon}: \Delta_{N-1}\to \{ \mu_{x^{\varepsilon}} \, |\, x^{\varepsilon} \ \mbox{is a Birkhoff solution to}\ (\ref{FKepsilon2})\ \mbox{of rotation vector}\ \omega\}\subset  M(\R/\Z)\, .$$
To define this map more precisely, let $0\leq \sigma_1,...,\sigma_N<1$ be the distinct local minima of $V$. As was shown in Example \ref{ex1}, there exists for every $p\in \Delta_{N-1}$ precisely one (up to translation) left-continuous hull function $\phi_p^-:\R\to\R$ such that $|(\phi_p^-)^{-1}(\sigma_j)|=p_j$. It generates the lamination
%We can use it to define a Birkhoff configuration $x^p:\Z^d\to \R$ by setting $x^p_i:= \phi_p^{-}(\omega\cdot i)$ that generates the lamination
$$\Gamma(x^p)=\{ (x^p)^{\pm}(s)\, |\, (x^p)^{\pm}(s)_i=\phi_p^{\pm} (s+\omega\cdot i)\}\ \mbox{where}\  x^p_i:= \phi_p^{-}(\omega\cdot i)\, .$$
Since each element of $\Gamma(x^p)$ takes values in $\{\sigma_1, \ldots, \sigma_N\}+\Z$, this lamination consists entirely of local minimizers of (\ref{antilimit}). Hence, if $0< \varepsilon \leq \varepsilon_1$, Theorem \ref{corollaryordering} guarantees that the elements of $\Gamma(x^p)$ continue to form a lamination $\Gamma(x^{p}(\varepsilon))$ of local minimizers of (\ref{FKepsilon2}), which in turn has its own hull functions $\phi^{\pm}_{x^{p}(\varepsilon)}$ and Radon probability measure $\mu_{x^{p}(\varepsilon)}$. 

The map $\Psi_\varepsilon:\Delta^{N-1}\to M(\R/\Z)$ is given by
\begin{align}\label{psieps}
\Psi_\varepsilon(p):= \mu_{x^p(\varepsilon)}\, .
\end{align}
Theorem \ref{homeothm} below is the precise version of Theorem \ref{mainthm}.  %We start with the following lemma.

\begin{theorem}\label{homeothm} Let $N$ be the number of geometrically distinct local minima of $V$. Then there exists an $\varepsilon_1>0$ so that for all $0< \varepsilon \leq \varepsilon_1$ the map
$$\Psi_\varepsilon: \Delta_{N-1} \to \{ \mu_{x^{\varepsilon}} \, |\, x^{\varepsilon} \ \mbox{is a Birkhoff solution to}\ (\ref{FKepsilon2})\ \mbox{of rotation vector}\ \omega\}\subset M(\R/\Z)$$ defined in (\ref{psieps}) is a homeomorphism.
\end{theorem}
\begin{proof}
Let us choose $0<\varepsilon_1\leq \varepsilon_0$ to be the minimum of the $\varepsilon_1$'s in the statements of Theorems \ref{corollaryordering}, \ref{conversethm} and \ref{birkhoffcontinuation} and assume that $0<\varepsilon\leq \varepsilon_1$.

To prove that $\Psi_{\varepsilon}$ is injective, recall from Theorem \ref{continuation} that $|X(\varepsilon)_i-X_i|< \delta_0$ for all $i\in \Z^d$, all $0\leq \varepsilon\leq \varepsilon_0$, all $X\in \Gamma(x^p)$ and all $p\in \Delta_{N-1}$. Because $(x^p)^-(s)_0 = \phi_{x^p}^-(s)$ and $(x^p(\varepsilon))^-(s)_0 = \phi_{x^p(\varepsilon)}^-(s)$ it follows that 
$$|\phi^-_{x^p(\varepsilon)}(s)- \phi^-_{x^p}(s)| < \delta_0\ \mbox{for all}\ s\in \R/\Z\, .$$ 
Let us now choose for every $1\leq j \leq N$ a continuous function $f_j: \R/\Z\to\R$ so that 
$$f_j|_{(\sigma_{j}-\delta_0, \sigma_j+\delta_0)} = 1\ \mbox{and}\ f_j|_{(\sigma_{k}-\delta_0, \sigma_k+\delta_0)} = 0\ \mbox{for all}\ k\neq j\, .$$ 
It is clear from (\ref{disjointintervals}) that such functions exist. It follows that 
$$\mu_{x^p(\varepsilon)} (f_j) = \int_{\R/\Z} f_j(\phi^-_{x^p(\varepsilon)}(s))ds = \int_{\R/\Z} f_j(\phi^-_{x^p}(s))ds  = p_j\, .$$ 
In particular, $\mu_{x^{p^1}(\varepsilon)}\neq \mu_{x^{p^2}(\varepsilon)}$ if $p^{1}\neq p^{2}$, so $\Psi_{\varepsilon}$ is injective.
 %Moreover, it holds that $\phi_p(y)=\phi^+_{x(p)}(y)$ for all $y\in \R/\Z$. In other words, the map $p\mapsto \Gamma(x(p))$ is well defined. 

%On the other hand, if $x$ is a recurrent Birkhoff configuration of rotation vector $\omega\in \R^d\backslash \Q^d$, such that for all $i\in \Z^d$ $x_i=\sigma_j$ for some $j\in \{1,...,N\}$, then by Proposition \ref{propertieshull} there exist corresponding hull functions $\phi^\pm_x$, which determine a $p\in \Delta_{N-1}$ by $|\phi_x^{-1}(\sigma_j)|=p_j$. This implies that the map $\Gamma(x)\mapsto p$ is well defined for such configurations $x$.

%Hence, if $x$ denotes an arbitrary recurrent Birkhoff solution of (\ref{antilimit}), such that all its coordinates lie in local minima and it has a fixed rotation vector $\omega\in \R^d\backslash \Q^d$, there is a one-to-one correspondence between $\Gamma(x)$, and the simplex $\Delta_{N-1}$. By Theorem \ref{ordered recurrence}, it holds for all $\varepsilon<\varepsilon_0$ that any such solution $x$ continues to a solution $x(\varepsilon)$, for which $\Gamma(x(\varepsilon))$ is a lamination. 

Next, Theorems \ref{conversethm} and \ref{birkhoffcontinuation} together imply that if $x^{\varepsilon}$ is any recurrent Birkhoff solution to (\ref{FKepsilon2}) of rotation vector $\omega$, then $x^{\varepsilon} = x(\varepsilon)$ for some recurrent Birkhoff local minimizing solution $x$ to (\ref{antilimit}) of rotation vector $\omega$. This implies in particular that $x_i\in \{\sigma_1, \ldots, \sigma_N\}+\Z$ for all $i\in\Z^d$ and as a consequence also the hull functions $\phi^{\pm}_x$ take values only in $\{\sigma_1, \ldots, \sigma_N\}+\Z$. Hence, $\phi^{\pm}_x = \phi^{\pm}_{x^p}(\cdot +s)$ for some $p\in \Delta_{N-1}$ and some $s\in \R$. By Theorem \ref{clas2} it thus follows that $\Gamma(x)=\Gamma(x^p)$. Moreover, Theorem \ref{corollaryordering} implies that $\Gamma(x^p(\varepsilon))$ is ordered (in fact a lamination by Proposition \ref{Cantorcontinuation}) and as a result,  $\Gamma(x(\varepsilon))=\Gamma(x^p(\varepsilon))$. In particular $\mu_{x^{\varepsilon}}=\mu_{x(\varepsilon)} = \mu_{x^p(\varepsilon)}$. Thus, $\Psi_{\varepsilon}$ is surjective.
%To say something about the continuity of this map, we define a topology on the class of laminations via the vague topology on the corresponding measures. More precisely, it holds by Proposition \ref{cross2} that for any Birkhoff recurrent configuration $x$ of a fixed rotation vector $\omega \in \R^d\backslash \Q^d$, the set $\Gamma(x)$ uniquely corresponds to the measure $\mu_{x}\in M(\R/\Z)$. With this identification, we define a topology on the set of laminations of rotation vector $\omega$ as the topology induced by the vague topology on $M(\R/\Z)$.

It remains to prove that $\Psi_{\varepsilon}$ is a homeomorphism. So let us assume that $\lim_{n\to\infty} p^n=p$ for certain $p^1, p^2, \ldots, p\in \Delta^{N-1}$ and let us denote by $\phi_p^-:\R\to \R$ the unique left-continuous hull function for which $|(\phi_p^-)^{-1}(\sigma_j)|=p_j$ and $\lim_{s\nearrow 0}\phi^-_p(s)\leq 0$ and $\lim_{s\searrow 0}\phi^-_p(s)> 0$. Then it holds that $\lim_{n\to\infty} \phi^-_{p^n} = \phi^-_p$ at any point of continuity of $\phi_p^-$, as was shown in Example \ref{ex1}. Let $t\in \R$ be any point with the property that $\phi^-_{p}$ is continuous at every point $t+\omega\cdot i$ with $i\in\Z^d$. Such $t$ exists because otherwise $\phi_p^-$ would have uncountably many discontinuities. We can then define $x^{p^n}_i:=\phi_{p^n}^-(t+\omega\cdot i)$ and $x^p_i:=\phi^-_{p}(t+\omega\cdot i)$ and by construction, $\lim_{n\to \infty} x^{p^n} = x^p$ pointwise. 

By Theorem \ref{continuous} it then holds that also $x^{p^n}(\varepsilon)\to x^p(\varepsilon)$ pointwise and therefore by Theorem \ref{prevaguetheorem} and Theorem \ref{vaguetheorem} we may conclude that $\mu_{x^{p^n}(\varepsilon)} \to \mu_{x^p(\varepsilon)}$ vaguely. So $\Psi_{\varepsilon}$ is continuous. %Since $x(p^n)_i=\phi_{p^n}(\omega \cdot i)$ and $x(p)_i=\phi_{p}(\omega \cdot i)$, have the same hull functions as $x^n$ and $x$, it follows that $\Gamma(x(p^n))=\Gamma(x^n)$ and $\Gamma(x(p))=\Gamma(x)$, so also $\Gamma(x(p^n)(\varepsilon))=\Gamma(x^n(\varepsilon))$ and $\Gamma(x(p)(\varepsilon))=\Gamma(x(\varepsilon))$. This shows that $\mu_{x(p^n)(\varepsilon)}=\mu_{x^n(\varepsilon)}\to\mu_{x(\varepsilon)}=\mu_{x(p)(\varepsilon)}$ vaguely.

It follows that $\Psi_{\varepsilon}$ is a homeomorphism because it is a continuous bijection from a compact topological space into a topological Hausdorff space. 
%On the other hand, assume that $\mu^n \to \mu$ in $\Phi^\varepsilon(\Delta_{N-1})\subset M(\R/\Z)$. Then there exist $p^1,p^2,...$ and $p$ in $\Delta_{N-1}$, such that $\mu_n$ and $\mu$ correspond to hull functions $\phi_{p^n(\varepsilon)}$ and $\phi_{p(\varepsilon)}$. As explained above, it then holds that $\mu_n=\mu_{x(p^n)(\varepsilon)}$ and $\mu=\mu_{x(p)(\varepsilon)}$ and that $\mu_{x(p^n)(\varepsilon)}\to\mu_{x(p)(\varepsilon)}$ in $M(\R/\Z)$. It then holds for all $f\in C^0(\R/\Z)$ that $$\int_{\R/\Z} f d\mu_{x(p^n)(\varepsilon)} = \int_{\R/\Z} f d\mu_{x(p)(\varepsilon)}\, .$$ In particular, if we choose $f_j$ to be such that $f_j(s)=1$ for all $s\in [\sigma_j-\delta_0, \sigma_j+\delta_0]$ and $f_j(s)=0$ for all $s\in [\sigma_i-\delta_0, \sigma_i+\delta_0]$, where $i\neq j$. Then it follows that $$\int_{\R/\Z} f_j d\mu_{x(p^n)(\varepsilon)}=p^n_j=\int_{\R/\Z} f_j d\mu_{x(p^n)},$$ because by construction of continuation, it holds that $|x(p)-x(p)(\varepsilon)|_{sup}\leq \delta_0$, which implies that the statistical occurence of $x(p)$ in $\sigma_j$, correspond to that of $x(p)(\varepsilon)$ in $[\sigma_i-\delta_0, \sigma_i+\delta_0]$. Hence, $p_j^n$ have the limit $p_j$, because of uniqeuness. 
\end{proof}

\appendix
\section{More proofs}\label{proofappendix}
In this appendix, we  
provide the proofs of Propositions \ref{numbertheory} and \ref{propertieshull} and Theorems \ref{xplusminuslemma} and \ref{Cantortheorem}. In addition, we formulate and prove the maximum principle and the comparison principle satisfied by equation (\ref{FKepsilon2}). All these results are standard, but complete proofs are not always easy to find in the literature.

 \begin{proofof}{Proposition \ref{numbertheory}}
Denote by $x^{\omega}$ the linear configuration defined by $x^{\omega}_i := x_0+\omega\cdot i$. Then $\tau_{k,l}x^{\omega}-x^{\omega} =\omega\cdot k + l$. Suppose for instance that $\omega\cdot k + l>0$, that is that $\tau_{k,l}x^{\omega} > x^{\omega}$, but assume on the other hand that $\tau_{k,l}x \leq x$. This means that $\tau_{k,l}x-x \leq 0$ and hence also that $\tau_{k,l}^2x-\tau_{k,l}x = \tau_{k,0}(\tau_{k,l}x-x) \leq 0$. Thus, $\tau_{k,l}^2x- x = (\tau_{k,l}^2x - \tau_{k,l}x)+(\tau_{k,l}x -x )\leq 0 $, i.e. $\tau^2_{k,l}x\leq x$. By induction we then find that
$\tau_{k,l}^n x \leq x$, for every $n\geq 1$.  On the other hand, $ \tau_{k,l}^n x^{\omega} = x^{\omega} + n(\omega\cdot k + l)$. This contradicts the fact that $\sup_i |\tau_{k,l}^n(x^{\omega}-x)_i| = \sup_i |(x^{\omega}-x)_i| \leq 1$ is uniformly bounded in $n$.
\end{proofof}

\begin{proofof}{Proposition \ref{propertieshull}}
We will prove Proposition \ref{propertieshull} for $\phi_x^-$. For $\phi_x^+$ the proofs are similar.

{\it 1)} Let $N$ be any integer larger than or equal to $s$. When $\omega\cdot k_n+l_n\nearrow s$, then for all $n\in \N$ there is an $M\in\N$ such that for all $m\geq M$ it holds that $\omega\cdot k_n+l_n \leq \omega\cdot k_m+l_m <s\leq \omega\cdot 0+ N$ and hence, by Proposition \ref{numbertheory}, that $x_{k_n}+l_n=(\tau_{k_n,l_n}x)_0\leq(\tau_{k_m,l_m}x)_0= x_{k_m}+l_m\leq x_0+N$. This implies that $\lim_{n\to\infty}x_{k_n}+l_n$ exists and is equal to $\limsup_{n\to\infty}x_{k_n}+l_n$.

Next, assume that $\omega\cdot k_n+l_n \nearrow s$ as $n\to\infty$ and $\omega\cdot K_m+L_m \nearrow s$ as $m\to\infty$. This implies that for every $n\in \N$ there is an $m\in \N$ so that $\omega\cdot K_m+L_m\geq \omega\cdot k_n+ l_n$. With Proposition \ref{numbertheory} this gives that $x_{K_m}+L_m\geq x_{k_n}+l_n$ and hence we find that $\limsup_{m\to\infty}x_{K_m}+L_m\geq \limsup_{n\to\infty}x_{k_n}+l_n$. Reversing the roles of the two sequences, we also find that $\limsup_{m\to\infty}x_{K_m}+L_m\leq \limsup_{n\to\infty}x_{k_n}+l_n$. Thus, $\phi_x^-$ is well-defined.

{\it 2)} If $\omega\cdot k_n+l_n\nearrow s$ and $\omega\cdot K_m+L_m\searrow s$, then $\omega\cdot k_n+l_n< \omega\cdot K_m+L_m$ for all $n$ and $m$ and thus, $\phi_x^-(s)=\lim_{n\to\infty}x_{k_n}+l_n \leq \lim_{m\to\infty}x_{K_m}+L_m =\phi_x^+(s)$. 

If $\omega\cdot k_n+l_n\nearrow 0$, then $\omega\cdot (k_n+i) + l_n \nearrow \omega\cdot i$ and $\tau_{k_n,l_n}x\leq x$ for all $n$ and therefore $\phi_x^-(\omega\cdot i) =\lim_{n\to\infty} x_{k_n+i}+l_n \leq x_i$.

{\it 3)} If $s_1 < s_2$ and $\omega\cdot k_n+l_n \nearrow s_1$ and $\omega\cdot K_m+L_m \nearrow s_2$, then there is an $M$ so that for all $n$ and all $m\geq M$ it holds that $\omega\cdot k_n+l_n<\omega\cdot K_m+L_m$. This clearly implies that $\phi_x^-(s_1)\leq\phi_x^-(s_2)$.  

{\it 4)} If $\omega\cdot k_n + l_n\nearrow s$, then $\omega\cdot k_n+ (l_n+1)\nearrow s+1$. This implies that $\phi_x^-(s+1)=\lim_{n\to\infty}x_{k_n}+l_n+1=\phi^-_x(s)+1$.  

{\it 5)} If $\omega\cdot k_n +l_n\nearrow s$, then $\omega\cdot (k_n+k)+(l_n+l)\nearrow s+\omega\cdot k+l$ and therefore $\phi_{\tau_{k,l}x}^-(s)=\lim_{n\to\infty} (\tau_{k,l}x)_{k_n}+l_n = \lim_{n\to\infty}x_{k+k_n}+(l+l_n)= \phi_x^-(s+\omega\cdot k+l)$.

{\it 6)} Assume that $s^n\nearrow s$ as $n\to\infty$. Because $\phi_x^-$ is nondecreasing, it follows immediately that $\lim_{n\to\infty}\phi_x^-(s^n)\leq \phi_x^-(s)$. To prove the opposite inequality, let $(k_m^n, l_m^n)$ be sequences of integers for which $\omega\cdot k_m^n+l_m^n\nearrow s^n$ as $m\to\infty$, so that $\phi^-_x(s^n)=\lim_{m\to\infty}x_{k_m^n}+l_m^n$.  We then have that for every $n$ there is an $m(n)$ so that $\omega\cdot k_{m(n)}^n+l_{m(n)}^n\nearrow s$, and thus 
\begin{align}
\lim_{n\to\infty}\phi_x^-(s^n) = \limsup_{n\to\infty}\phi_x^-(s_n)=\limsup_{n\to\infty}\limsup_{m\to\infty} x_{k_m^n}+l^n_m \geq \limsup_{n\to\infty}x_{k_{m(n)}^n}+l_{m(n)}^n = \phi_x^-(s)\ . \nonumber %\label{fancylimit}
\end{align}
Thus, $\phi_x^-$ is left-continuous.

{\it 7)} Let $s_n\searrow s$. First of all, because $\phi_x^-\leq \phi_x^+$, we have that $\phi_x^-(s_n)\leq \phi_x^+(s_n)$. Because $\phi_x^+$ is right-continuous, this implies that $\lim_{n\to\infty}\phi_x^-(s_n)\leq \phi_x^+(s)$. To prove the other inequality, let us denote again by $(k_m^n, l_m^n)$ sequences of integers with $\omega\cdot k_m^n +l_m^n\nearrow s^n$. We then have that for every $n$ there is an $m(n)$ so that $\omega\cdot k_{m(n)}^n+l_{m(n)}^n\searrow s$, and thus
\begin{align}
\lim_{n\to\infty} \phi_x^-(s_n)=\liminf_{n\to\infty}\phi_x^-(s^n)=\liminf_{n\to\infty}\limsup_{m\to\infty} x_{k_m^n}+l^n_m \geq \liminf_{n\to\infty}x_{k_{m(n)}^n}+l_{m(n)}^n = \phi_x^+(s)\ . \nonumber %\label{fancylimit2}
\end{align}
This proves {\it 7)}.

{\it 8)} Assume that $\phi_x^-$ is right-continuous at some point $t\in\R$. Then it holds for any sequence $s^n\searrow t$ that $\phi_x^-(t)=\lim_{n\to\infty}\phi_x^-(s^n) = \phi_x^+(t)$ by property {\it 7)}. Thus, $\phi_x^-=\phi_x^+$ at all points of continuity of $\phi_x^-$. Similarly, if $\phi_x^-$ is not right-continuous at $t\in\R$, then $\phi_x^-(t)\neq \lim_{n\to\infty}\phi_x^-(s^n) = \phi_x^+(t)$.

%{\it 9)} It holds that $\phi_x^-(y+\omega\cdot k + l) = \phi_{\tau_{k,l}x}(y) = \lim_{n\to\infty}$ ????

{\it 9)} Because $\phi_x^-$ is nondecreasing, it has only countably many discontinuities. Thus it follows from {\it 8)} that $\phi_x^-  = \phi_x^+$ almost everywhere.
\end{proofof}

\begin{proofof}{Theorem \ref{xplusminuslemma}}
%We will first prove that $\{x^{\pm}(y)\, | \, y\in\R\}$ is shift-invariant, closed, recurrent and minimal. 
It is clear that $\Gamma(x)$ is nonempty. It follows from properties {\it 2)}, {\it 3)} and {\it 7)} of Proposition \ref{propertieshull} that when $s_1< s_2$, then $\phi_{x}^-(s_1+\omega\cdot i) \leq \phi_{x}^+(s_1+\omega\cdot i) \leq \phi_{x}^-(s_2+\omega\cdot i)\leq \phi_{x}^+(s_2+\omega\cdot i)$. As a consequence,
$$x^-(s_1)\leq x^+(s_1) \leq x^-(s_2)\leq x^+(s_2)\ \mbox{for}\ s_1<s_2\, .$$ 
Thus, $\Gamma(x)$ is ordered. Shift-invariance follows from property {\it 4)}, namely
$(\tau_{k,l}x^{\pm}(s))_i = x^{\pm}(s)_{i+k}+l = \phi_x^{\pm}(s+\omega\cdot (i+k))+l =\phi_x^{\pm}(s+\omega\cdot k+l+\omega\cdot i) =x^{\pm}(s +\omega\cdot k+l)_i$.
Thus,
\begin{align}\label{taulimits}
\tau_{k,l}x^{\pm}(s)= x^{\pm}(s+\omega\cdot k+l)\, .
\end{align} 
So $\Gamma(x)$ is shift-invariant. Next, we remark that properties {\it 6)} and {\it 7)} imply that 
\begin{align}\label{zlimits}
 x^{\pm}(s^n)\to x^{-}(s)\ \mbox{if} \ s^n\nearrow s \ \mbox{and}\  x^{\pm}(s^n)\to x^{+}(s^n)\ \mbox{if} \ s^n\searrow s\, .
\end{align} 
This means that $\Gamma(x)$ is closed under pointwise convergence. Finally, (\ref{taulimits}) and (\ref{zlimits}) together imply that $\Gamma(x)$ is minimal: because $\omega\in\R^d\backslash \Q^d$, every element of $\Gamma(x)$ is a limit of appropriate translates of itself or any other element of $\Gamma(x)$. In particular, every element of $\Gamma(x)$ is recurrent.
\end{proofof}

\begin{proofof}{Theorem \ref{Cantortheorem}}
It was already proved in Theorem \ref{xplusminuslemma} that $\Gamma(x)$ is closed. Moverover, $\Gamma(x)$ is perfect because each of its elements is recurrent. 

Next, assume first of all that $\phi^-_x$ is continuous. Then $\phi_x^-=\phi_x^+$.  Moreover, if $s_n\to s$, then $x^-(s_n)_i=\phi_x^-(s_n+\omega\cdot i) \to \phi_x^-(s+\omega\cdot i) = x^-(s)_i$ and therefore $x^-(s_n)\to x^-(s)$ pointwise. Thus, the map $s\mapsto x^-(s)=x^+(s)$ is a continuous bijection from $\R$ to $\Gamma(x)$ that descends to a continuous map from $\R/\Z$ to $\Gamma(x)/\Z$. Because $\R/\Z$ is compact and $\Gamma(x)/\Z$ is Hausdorff, we conclude that $\Gamma(x)$ is homeomorphic to $\R$ and hence connected.

When $\phi_x^-$ is not continuous, let $t$ be a point of discontinuity, so that $\phi_x^-(t)<\phi_x^+(t)$. Then $x^-(t) < x^+(t)$, which implies that $\Gamma(x)$ is not connected. In fact, one calls the order interval $[x^-(t), x^+(t)] :=\{x\in \R^{\Z^d}\, |\, x^-(t)\leq x \leq x^+(t)\}$ a {\it gap} in $\Gamma(x)$: it contains no elements of $\Gamma(x)$ other than its boundaries $x^{\pm}(t)$ themselves. 

We will show that this implies that between any two elements of $\Gamma(x)$ there exists a gap. So let $s_1\leq s_2$ be any two numbers. When $s_1=s_2$ then either $x^-(s_1)=x^+(s_1)$ or $x^-(s_1)<x^+(s_1)$, whence $[x^-(s_1), x^+(s_1)]$ forms a gap. Thus, we may assume that $s_1<s_2$. But then there exist $k$ and $l$ such that $s_1<t+\omega\cdot k+l<s_2$. It follows that
$$x^+(s_1)< x^-(t+\omega\cdot k+l)\leq x^+(t+\omega\cdot k+l) < x^-(s_2)\, .$$
Now we recall from (\ref{taulimits}) that $x^{\pm}(t+\omega\cdot k+l)= \tau_{k,l}x^{\pm}(t)$. Hence, being the translate of a gap, $[x^-(t+\omega\cdot k+l), x^+(t+\omega\cdot k+l)]$ is a gap itself. This proves that there is a gap between $x^+(s_1)$ and $x^-(s_2)$. In particular, $\Gamma(x)$ is totally disconnected: it splits as the disjoint union of the closed sets $ \{X\in \Gamma(x) \ | \ X \leq x^-(t+\omega\cdot k + l)\}$ and $\{X\in \Gamma(x) \ | \ X\geq x^+(t+\omega\cdot k + l)\}$ that contain $x^+(s_1)$ and $x^-(s_2)$ respectively.
\end{proofof}

\begin{lemma}[Minimum-maximum property]\label{Aubry}
For $\varepsilon \geq 0$, $B\subset \Z^d$ finite and $x,y:\Z^d\to \R$ arbitrary,
\begin{align}
W^\varepsilon_B(x \wedge y)+W^\varepsilon_B(x \vee y) \leq W^\varepsilon_B(x) + W^\varepsilon_B(y) \ .
\end{align}
\end{lemma}
\begin{proof}
Let us write $\alpha:=(x\wedge y)-x$ and $\beta:=(x\vee y)-x$ and observe that $\alpha \leq 0$ and $\beta\geq 0$, while $\mbox{supp}(\alpha) \cap \mbox{supp}(\beta) = \emptyset$ and $y=(x\wedge y)+(x\vee y)-x= (x\wedge y) +\beta=\alpha + \beta + x$. Thus, 
\begin{align}\nonumber 
& W_B^{\varepsilon}(x)+W_B^{\varepsilon}(y)- W_{B}^{\varepsilon}(x\wedge y)-W_{B}^{\varepsilon}(x\vee y) = \\ \nonumber
 W_{B}^{\varepsilon}&(x) + W_{B}^{\varepsilon}(x+\alpha+\beta)-W_{B}^{\varepsilon}(x+\alpha)-W_{B}^{\varepsilon}(x+\beta)\, .
\end{align}
This expression can be put in integral form as
\begin{align}
 \nonumber
&  \int_0^1\int_0^1\frac{\p^2}{\p t\p s} W_{B}^{\varepsilon}(x+\alpha t+\beta s)ds dt = \\ \nonumber 
 \sum_{i,k\in \Z^d}&\sum_{j \in B}\left( \int_0^1\int_0^1 \partial_{i,k} S_j^{\varepsilon}(x + t\alpha+ s\beta)dsdt \right) \alpha_i\beta_k \, .
\end{align}
Here we wrote $S_j^{\varepsilon}(x)=V(x_j)+\varepsilon S_j(x)$. Since $\mbox{supp}(\alpha) \cap \mbox{supp}(\beta) = \emptyset$, we have that $\alpha_i\beta_i=0$ for all $i$. Moreover, the twist condition $\partial_{i,k}S_{j}\leq 0$ for all $i\neq k$ and the inequalities $\alpha_i\beta_k\leq 0$, guarantee that this expression is nonnegative. This proves the lemma.
\end{proof}
%begin{proof}
%For a proof of this simple fact, see \cite{MramorRink1}. %If now $x$ and $y$ are global minimizers and $v$ has finite support, then one can choose $B$ slightly larger than this support, namely so that for all $i\in \sup v$ it holds for all $j\in \Z^d$ with $||j-i||\leq 1$ that $j\in B$. For such $B$ it holds that 
%$$W^{\varepsilon}(x+v)-W^{\varepsilon}(x)=W^{\varepsilon}_B(x+v)-W^{\varepsilon}_B(x)$$
%and similarly with $x$ replaced by $y, x\wedge y$ and $x\vee y$.
%Thus, in view of the fact that $x\wedge y + v = (x+v)\wedge (y+v)$ and $x\vee y + v = (x+v)\vee (y+v)$, formula (\ref{Aubry}) just says that 
%$$W^{\varepsilon}(x\wedge y+v) - W^{\varepsilon}(x\wedge y) + W^{\varepsilon}(x\vee y+v) - W^{\varepsilon}(x\vee y) $$
%\end{proof} 
%Another standard result is a strong maximum principle, which we also state without a proof:
\begin{lemma}[Comparison principle]\label{maximumprinciple}
Let $\varepsilon>0$, $B\subset \Z^d$ finite and assume that $\mathring{B}$ is path-connected. Moreover, let $x, y:\Z^d\to\R$ be two configurations that satisfy  
$$ \p_iW_B^{\varepsilon}(x) = 0  \ \mbox{and}\ \p_iW_B^{\varepsilon}(y) = 0\ \mbox{for all}\ i\in \mathring{B}\, .$$ 
If $x_i\leq y_i$ for all $i\in \overline{B}$, then either $x_i = y_i$ for all $i\in \mathring{B}$ or $x_i < y_i$ for all $i\in \mathring{B}$.
\end{lemma} 
\begin{proof}
By contradiction: assume that $x_i\leq y_i$ for all $i\in \overline{B}$ and that there are indices $k, l\in \mathring{B}$ such that $x_k=y_k$ and $x_l<y_l$. Because $\mathring{B}$ is path-connected, it can be assumed that $||k-l||=1$. Now we compute
\begin{align}\nonumber
 \p_kW_B^{\varepsilon}(y) -   \p_kW_B^{\varepsilon}(x) & =  \sum_{j\in B}\left(\p_kS^{\varepsilon}_j(y) - \p_kS_j^{\varepsilon}(x)\right) = \\  \sum_{j\in B}\int_0^1 \frac{d}{dt} \partial_kS_j^{\varepsilon}(ty+(1-t)x) dt = & \!\!\!\!\!\!\! \sum_{{\tiny \begin{array}{c} j\in B\\ ||i-j||\leq r\end{array}}} \!\!\!\!\!\!\!  \left( \int_0^1\partial_{i,k} S^{\varepsilon}_j(ty + (1-t)x)dt \right)  (y_i - x_i)\, . \nonumber
\end{align}
Recall that for every $i\neq k$, it holds that $\p_{i,k}S^{\varepsilon}_j\leq 0$. Because $(y_i-x_i)\geq 0$ for all $i\in \overline{B}$ by assumption and $y_k-x_k=0$, every term in the above sum is nonpositive. But for the $l$ chosen above, $\partial_{k,l}S^{\varepsilon}_k=\varepsilon \p_{k,l}S_k<0$, while $y_l-x_l>0$. This proves that $\p_kW_B^{\varepsilon}(y) - \p_kW_B^{\varepsilon}(x)<0$, so $x$ and $y$ are not both stationary configurations.
\end{proof}
Lemma \ref{maximumprinciple} implies in particular that if $\varepsilon>0$ and $x$ and $y$ are solutions to (\ref{FKepsilon2}) and $x_i\leq y_i$ for all $i\in \Z^d$, then either $x_i=y_i$ for all $i\in \Z^d$ or $x_i<y_i$ for all $i\in \Z^d$.

\section{The standard map near an anti-integrable limit} \label{appendix}
%In case $n=1$, the recurrence relation (\ref{FKmap}) reduces to the well-known $1$-dimensional Frenkel-Kontorova problem
%\begin{align}\label{FKdynamic}
%\varepsilon (x_{i+1}-2x_i+x_{i-1}) - V' (x_i)\ \mbox{for}\ x_i\in\R \ \mbox{and} \ i\in\Z\ .
%\end{align}
%This recurrence relation describes the stationary states of a chain of atoms on a periodic substrate with linear attraction between nearest neighbors. The special case that $\varepsilon=0$ in (\ref{FKdynamic}) is sometimes referred to as the {\it anti-continuum limit} of the lattice. 
%This means that the result of \cite{baesens}  classifies the well-ordered equilibrium states with irrational average atomic displacement of certain $1$-dimensional ferromagnetic nearest-neighbor crystals close to a nondegenerate anti-integrable limit.
In this appendix we describe some of the classical results on Hamiltonian twist maps near a nondegenerate anti-integrable limit. More on the application of Aubry-Mather theory in the theory of twist maps can be found in the classical references \cite{angenent88}, \cite{gole91}, \cite{gole92}, \cite{gole01},  \cite{MatherTopology} and \cite{MatherForni}. 

Let us streamline our exposition by discussing only one canonical example: the standard map, the quotient to $\R^n / \Z^n \times\R^n$ of the map
\begin{align}\label{smap}
T_{\varepsilon}: (x,y)\mapsto (x+y+\varepsilon^{-1} \nabla V(x), y+\varepsilon^{-1} \nabla V(x)) \ \mbox{on} \ \R^n\times\R^n.
\end{align}
Here $V:\R^n\to \R$ is a $\Z^n$-periodic Morse function, i.e. $V$ is twice continuously differentiable and its critical points are nondegenerate. 

It is straightforward to check, for $\varepsilon>0$, that a sequence $i\mapsto(x_i, y_i)\in\R^n\times \R^n$ is an orbit of $T_{\varepsilon}$ if and only if $y_i=x_i-x_{i-1}$ and 
\begin{align}\label{FKmapn}
\nabla V(x_i)-\varepsilon(x_{i+1}-2x_i+x_{i-1}) =0\ \mbox{for all} \ i\in \Z\, .
\end{align}
Letting $\varepsilon\to \infty$, the map $T_{\varepsilon}$ reduces to the integrable twist map $(x, y)\mapsto (x+y, y)$. This explains why the limit $\varepsilon\downarrow 0$ is sometimes called the {\it anti-integrable limit} of the standard map. In contrast with the map $T_{\varepsilon}$, the recurrence relation (\ref{FKmapn}) is well-defined for $\varepsilon=0$. Its solutions are all sequences $x=(\ldots, x_{-1}, x_0, x_1, \ldots)$ of critical points of $V$. By a result similar to Theorem \ref{continuation}, those sequences of critical points for which $||x_{i+1} - x_i||$ is uniformly bounded, can be continued to solutions $x(\varepsilon)$ of (\ref{FKmapn}) for $0<\varepsilon\ll 1$. This yields many interesting orbits of $T_{\varepsilon}$ near its anti-integrable limit.

By prescribing $y_i=x_i-x_{i-1}$ bounded but more or less at random, one produces for example orbits $(x_i(\varepsilon), y_i(\varepsilon))$ of $T_{\varepsilon}$ with chaotic transitions in the momenta $y_i(\varepsilon)$. This was shown in \cite{abramovici}, in which the idea of an anti-integrable limit was first introduced.  On the contrary in \cite{MackayMeiss}, orbits of $T_{\varepsilon}$ with arbitrary rotation vector $\omega=\lim_{i\to\pm\infty}\frac{x_i}{i}\in \R^n$ were constructed by continuation from the anti-integrable limit. This result implies that $T_{\varepsilon}$ possesses invariant Cantor sets of all irrational rotation vectors. 

In the special case that $n=1$ equation (\ref{FKmapn}) reduces to the one-dimensional variant of the familiar Frenkel-Kontorova problem
\begin{align}\label{FKmap1}
V'(x_i)-
\varepsilon(x_{i+1}-2x_i+x_{i-1}) =0\ \mbox{for all} \ i\in \Z\, .
\end{align}
Theorem \ref{mainthm} says that for every $\omega\in\R\backslash \Q$ there exists an $N-1$-dimensional family of laminations of solutions to (\ref{FKmap1}), where $N$ is the number of geometrically distinct local minima of $V$. Each of these laminations is of the form 
$$\Gamma(x(\varepsilon))=\Gamma_{\varepsilon}(x)=\{(\ldots, X_{-1}(\varepsilon), X_0(\varepsilon), X_1(\varepsilon), \ldots)\, |\, X\in \Gamma(x)\}\subset \R^{\Z}$$
for some Birkhoff sequence $x:\Z\to\R$ of local minima of $V$. 
 
 In turn, any such lamination corresponds to a $T_{\varepsilon}$-invariant Cantor subset 
 $$C_{\varepsilon}(x):= \{ (X_{0}(\varepsilon), Y_0(\varepsilon)) = (X_{0}(\varepsilon), X_0(\varepsilon)-X_{-1}(\varepsilon)) \, | \, X\in \Gamma(x)   \}\subset \R\times \R\, .$$
An invariant Cantor set $C_{\varepsilon}(x)\subset \R\times\R$ with the property that $\Gamma_{\varepsilon}(x)$ is a lamination, is sometimes called a {\it remnant circle} or {\it cantorus}, where the latter terminology was introduced by Percival \cite{percival2}. This terminology is motivated by the well-known theorem of Birkhoff that the orbits in a so-called rotational invariant circle for $T_{\varepsilon}$ in fact always form a foliation. Theorem \ref{mainthm} implies that the collection of remnant circles of $T_{\varepsilon}$ of irrational rotation number $\omega$ is homeomorphic to $\Delta_{N-1}$. This latter result is precisely the one described in \cite{baesens}. 

At the same time, the collection of invariant Cantor sets of $T_{\varepsilon}$ is actually much larger. This follows from Theorem \ref{disone} below, which we prove in case $n=1$. The same result is essentially contained in \cite{MackayMeiss}, although the consequences that are mentioned in the latter paper do not seem completely justified.
 \begin{theorem}\label{disone}
Let $x:\Z\to\R$ be any Birkhoff sequence of critical points (i.e. not necessarily local minima) of $V:\R\to \R$ with rotation number $\omega\in \R\backslash\Q$. Then the collection 
$$\{(X_0(\varepsilon), X_0(\varepsilon) - X_{-1}(\varepsilon))\ |\ X\in\Gamma(x)\}\subset\R^2$$ is an invariant Cantor set for $T_{\varepsilon}$.    
\end{theorem}

\begin{proof}
Because $x$ takes discrete values, $\Gamma(x)$ is disconnected and hence by Theorem \ref{Cantortheorem} a Cantor subset of $\R^{\Z}$. Theorem \ref{continuous} thus implies that so is $\Gamma_{\varepsilon}(x):=\{X(\varepsilon)\, | \, X\in \Gamma(x)\}$. At the same time, $\Gamma_{\varepsilon}(x)$ consists of solutions to (\ref{FKmap1}) and hence of orbits of the two-dimensional map 
$$t_{\varepsilon}: (X_{i-1}(\varepsilon), X_i(\varepsilon))\mapsto (X_i(\varepsilon), 2X_i(\varepsilon) -X_{i-1}(\varepsilon)+ \varepsilon^{-1}V'(X_i(\varepsilon)))\, . $$ 
This observation makes it clear that the projection $\pi:X(\varepsilon)\mapsto (X_0(\varepsilon), X_1(\varepsilon))$ from $\Gamma_{\varepsilon}(x)$ to $\R^2$ is injective. It is also clear that this projection is continuous. The inverse 
$$\pi^{-1}:(X_0(\varepsilon), X_1(\varepsilon))\mapsto (\ldots, X_{-1}(\varepsilon), X_0(\varepsilon), X_1(\varepsilon), X_2(\varepsilon), \ldots)$$ is continuous as well, because $(X_{i-1}(\varepsilon), X_{i}(\varepsilon))=t_{\varepsilon}^{i-1}(X_0(\varepsilon), X_1(\varepsilon))$ and because $t_{\varepsilon}^{i-1}$ is continuous for all $i\in\Z$. Thus, $\pi$ is a homeomorphism from $\Gamma_{\varepsilon}(x)$ onto its image in $\R^2$. 

The conjugacy $h:(X_{i-1}(\varepsilon), X_{i}(\varepsilon))\mapsto (X_i(\varepsilon), Y_i(\varepsilon)) = (X_i(\varepsilon), X_i(\varepsilon)- X_{i-1}(\varepsilon))$ between $t_{\varepsilon}$ and $T_{\varepsilon}$ maps this image homeomorphically to the $T_{\varepsilon}$-invariant set $\{(X_0(\varepsilon), X_0(\varepsilon) - X_{-1}(\varepsilon))\ |\ X\in\Gamma(x)\}\subset \R^2$.
\end{proof}
\noindent Theorem \ref{disone} implies that a Hamiltonian twist map near a nondegenerate anti-integrable limit has at least an $M-1$-dimensional family of invariant Cantor sets of  rotation number $\omega\in \R\backslash \Q$, where $M$ is the number of geometrically distinct critical points of $V$. Since $M$ is twice the number $N$ of geometrically distinct local minima of $V$, the results in this paper imply that most of these Cantor sets are not remnant circles.
\begin{small}
\bibliographystyle{amsplain}
\bibliography{anticontinuum}
\end{small}
 \end{document}